\documentclass[11pt]{amsart}

\usepackage{amssymb, amsmath, amsfonts, amsthm}
\usepackage{amscd}
\usepackage{color}
\usepackage{relsize}
\usepackage[all]{xy}
\numberwithin{equation}{section}
\usepackage{hyperref}
\usepackage{mathrsfs}

\usepackage{verbatim}
\usepackage[english]{babel}

\newtheorem{theorem}{{\textbf Theorem}}[section]
\newtheorem{prop}[theorem]{{\textbf Proposition}}
\newtheorem{cor}[theorem]{{\textbf Corollary}}
\newtheorem{lemma}[theorem]{{\textbf Lemma}}

\newtheorem{rmk}[theorem]{{\textbf Remark}}
\newenvironment{remark}{\begin{rmk}\rm}{\end{rmk}}
\newtheorem{definition}[theorem]{Definition}

\def \Z {\mathbb{Z}}
\def \C {\mathbb{C}}
\def \cB {\mathcal{B}}
\def \cD {\mathcal{D}}
\def \cR {\mathcal{R}}

\def \cB {\mathcal{B}}
\def \cC {\mathcal{C}}
\def \cP {\mathcal{P}}

\def \Qt{\widetilde{Q}}

\def \Rn{\cR(n)}
\def \Dn{\cD(n)}

\def \bn{\begin{enumerate}}
\def \en{\end{enumerate}}
\def \bdm{\begin{displaymath}}
\def \edm{\end{displaymath}}

\def \bp{\begin{proof}}
\def\ep{\end{proof}}

\def\U+e{(U^+)^e}

\def\be{\begin{equation}}
\def\ee{\end{equation}}

\def\lg{\mathrm{LG}(n)}
\def\In{\mathcal{I}_n}

\def\Sp{\mathrm{Sp}}

\def\cW{\mathcal{W}}

\def\cE{\mathcal{E}}

\def\bbY{\mathbb{Y}}
\def\bbX{\mathbb{X}}
\def\tW{\widetilde{W}}
\def\Mor{\mathrm{Mor}}
\def\rank{\mathrm{rk}\hspace{0.02in}}

\newcommand{\cO}{\mathcal O}
\newcommand{\Sym}{\mathrm{Sym}}
\newcommand{\Supp}{\mathrm{Supp}}
\newcommand{\Ker}{\mathrm{Ker}}
\newcommand{\Gr}{\mathrm{Gr}}
\newcommand{\Chow}{\mathrm{CH}}

\newcommand{\tN}{\widetilde{N}}
\newcommand{\bE}{\overline{E}}
\newcommand{\LQ}{LQ_e (W)}

\newcommand{\LQo}{{LQ}_e^\circ(W)}
\newcommand{\scW}{{\mathscr{W}}}
\newcommand{\scWb}{{\mathscr{W}|_b}}
\newcommand{\scWbo}{{\mathscr{W}|_{b_0}}}

\newcommand{\ev}{\mathrm{ev}}

\newcommand{\Oc}{{\mathcal O}_{C}}

\newcommand{\Quot}{\mathrm{Quot}}

\newcommand{\LG}{\mathrm{LG}}

\newcommand{\isom}{\xrightarrow{\sim}	}

\title[Counting maximal Lagrangian subbundles]{Counting maximal Lagrangian subbundles
over an algebraic curve}

\author{Daewoong Cheong}
\address{Chungbuk National University, Department of Mathematics, Chungdae-ro 1, Seowon-Gu, Cheongju City, Chungbuk 28644, Korea}
\email{daewoongc@chungbuk.ac.kr}

\author{Insong Choe}
\address{Department of Mathematics, Konkuk University, 1 Hwayang-dong, Gwangjin-Gu, Seoul 143-701, Korea}
\email{ischoe@konkuk.ac.kr}

\author{George H.\ Hitching}
\address{Oslo Metropolitan University, Postboks 4, St. Olavs plass, 0130 Oslo, Norway}
\email{gehahi@oslomet.no}

\begin{document}

\begin{abstract}
Let $C$ be a smooth projective curve and $W$ a symplectic bundle over $C$. Let $\LQ$ be the Lagrangian Quot scheme parametrizing Lagrangian subsheaves $E \subset W$ of degree $e$. We give a closed formula for intersection numbers on $\LQ$. As a special case, for $g \ge 2$, we compute the number of  Lagrangian subbundles of maximal degree of a general stable symplectic bundle, when this is finite. This is a symplectic analogue of Holla's enumeration of maximal subbundles in \cite{Ho}.
\end{abstract}

\maketitle

\section{Introduction}
Let $C$ be a smooth projective curve of genus $g \ge 2$, and $V$ a vector bundle of  rank $r$ and degree $d$ over $C$. For $1 \leq k \leq r-1$, a rank $k$ subbundle $E$ of $V$ is called a \textit{maximal subbundle} if $\deg(E)$ is maximal among all subbundles of rank $k$. Consider the following enumerative problem.

\smallskip
{\it What is the number of rank $k$ maximal subbundles of $V$, when it is finite?}
\smallskip

Classically, Segre \cite{Se} and Nagata \cite{Na} proved that if $d \not\equiv g \mod 2$, then a general stable bundle of rank two has $2^g$ maximal line subbundles. Later, Holla \cite{Ho} gave an explicit formula enumerating maximal subbundles in general (see also \cite{LN}, \cite{Ox} and \cite{Te}).

The goal of this article is to give an analogous result for symplectic bundles. To pose the problem, let us recall some basic notions. Let $L$ be a line bundle of degree $\ell$. An \textit{$L$-valued symplectic bundle} is a vector bundle $W$ on $C$ equipped with a nondegenerate skewsymmetric bilinear form $\omega \colon W \otimes W \rightarrow L$. Such a $W$ has rank $2n$ for some $n \ge 1$. From the induced isomorphism $W \cong W^\vee \otimes L$, we have $\deg (W) = n\ell$. In fact, it can be shown that $\det (W) \cong L^n$ (see \cite[\S\:2]{BG}).

A subsheaf $E \subset W$ is called \textit{isotropic} if $\omega ( E \otimes E ) = 0$. By linear algebra, $\rank(E) \le n$. If $\rank (E) = n$ then $E$ is said to be \textit{Lagrangian}. A \textit{maximal Lagrangian subbundle} of $W$ is one whose degree is maximal among all Lagrangian subsheaves of $W$.

Let $W$ be an $L$-valued symplectic bundle over $C$. For each integer $e$, let $\LQ$ be the Lagrangian Quot scheme parametrizing Lagrangian subsheaves $[E \to W]$ with $\deg(E) = e$; equivalently, quotients $[q \colon W \to F]$ with $F$ coherent of rank $n$ and degree $n\ell - e$, and $\Ker (q)$ isotropic. The scheme $\LQ$ is projective, and contains the quasiprojective subscheme $\LQo$ consisting of Lagrangian subbundles. By \cite[Proposition 2.4]{CCH}, the expected dimension of $\LQ$ is
\[ D ( n, e, \ell ) \ := \ -(n+1)e - \frac{n(n+1)}{2}(g-1-\ell) . \]
Based on results in \cite{CH2}, we will see the following.

\newtheorem*{prop-zero}{Proposition \ref{Prop-zero}}
\begin{prop-zero}
Let $L$ be a line bundle of degree $\ell$ and $W$ an $L$-valued symplectic bundle of rank $2n$ which is general in moduli. Write
\[ e_0 \ := \ \left\lceil -\frac{1}{2}n(g-1-\ell) \right\rceil . \]
\begin{enumerate}
\item A maximal Lagrangian subbundle of $W$ has degree $e_0$, and $LQ_{e_0}(W) = LQ_{e_0}^\circ(W)$.
\item If $n(g-1-\ell)$ is even, then $LQ_{e_0}(W)$ is a smooth scheme of dimension zero.
\end{enumerate}
\end{prop-zero}

\noindent This indicates how Lagrangian Quot schemes enter the picture. Our problem reduces to evaluating the integral
\[ \int_{[LQ_{e_0}(W)]} 1 \ \null_. \]
\smallskip
To compute this integral, more generally we find a closed formula for integrals $\int_{[\LQ]} P$, where $W$ is an arbitrary symplectic bundle, $e$ an integer and $P$ a certain cohomology class on $\LQ$. To obtain the desired formula, we follow essentially the method of Holla \cite{Ho} for the case of vector bundles. An important ingredient in the argument of \cite{Ho} is the fact, proven in \cite[{\S} 6]{PoRo}, that for small enough values of $e$, the scheme $\Quot_{k, e} ( W )$ subsheaves of rank $k$ and degree $e$ in $W$ is of the expected dimension, and that a general point of any component corresponds to a vector subbundle. For the present work, an analogous statement on Lagrangian Quot schemes is required. This follows from \cite{CCH}. (We mention that in both \cite{PoRo} and \cite{CCH} the respective Quot schemes are even shown to be irreducible.)

\bigskip

Let us give a sketch of the strategy for obtaining the formula. We begin with some terminology.

\begin{definition} We say that $\LQ$ \textit{has property $\cP$} if every component of $\LQ$ is generically smooth of the expected dimension $D ( n, e, \ell )$, and moreover a general point corresponds to a subbundle of $W$. \end{definition}

When $\LQ$ has property $\cP$, the fundamental class $[\LQ]$ is well behaved. In this case, $\int_{[\LQ]} P$ is an intersection number $N^w_{C, e} ( W ; P)$, which for general values of the parameters counts points lying in the saturated part $\LQo$. In general, however, $\LQ$ may have pathologies. Therefore we define another intersection number $\widetilde{N}^w_{C, e} ( W ; P)$, valid for any nonempty $\LQ$, as follows.

Firstly, we embed $\LQ$ in $LQ_e ( \tW )$ where $\tW$ is a symplectic Hecke transform of $W$ such that $LQ_e ( \tW )$ has property $\cP$. Then we set
\[ \widetilde{N}^w_{C, e} ( W ; P) \ := \ \int_{[ LQ_e ( \tW ) ]} P Q \]
for a suitable class $Q$. If the original $\LQ$ has property $\cP$, then $Q$ is just the class of the image of $\LQ \hookrightarrow LQ_e ( \tW )$. The key point, however, is that $Q$ is defined only in terms of a choice of Lagrangian subspaces of a finite number of fibers of $\tW$, and so makes sense even if $\LQ$ is not well behaved. Once $\widetilde{N}^w_{C, e} ( W ; P)$ is shown to be independent of the choice of $\tW$, it is straightforward to see that the two definitions of intersection number coincide when $\LQ$ has property $\cP$.

We then use this intersection theory to answer the enumerative problem stated at the outset. As the integral $\int_{[ \LQ ]} 1$ is intractable in this form, we follow \cite{Ho} and link $\tW$ with the trivial symplectic bundle $\Oc^{2n}$ by another sequence of Hecke transforms. Then the integral coincides with a genus $g$ Gromov--Witten invariant of the Lagrangian Grassmannian $\LG ( \C^{2n} )$. Using results from \cite{CMP}, \cite{Che1} and \cite{Che2}, the latter can be connected to a genus zero Gromov--Witten invariant of $\LG ( \C^{2n} )$, whose closed formula is given by a Vafa--Intriligator-type formula.

For $\LQ$ not having property $\cP$, this approach is an alternative to the use of virtual classes in developing an intersection theory, as done in \cite{MO} for the usual Quot schemes. It would be interesting to follow the approach of \cite{MO} for the Lagrangian Quot schemes.

\bigskip

The paper is organized as follows. In \S\:2, we review the quantum cohomology of Lagrangian Grassmannians. In \S\:3, we give basic properties of the Lagrangian Quot scheme $\LQ$ and discuss property $\cP$ and the nonsaturated locus. In \S\:4, we define Lagrangian degeneracy loci on $\LQ$ and investigate their properties and behavior under Hecke transforms. In \S\:5, we develop an intersection theory on $\LQ$ and find relations among intersection numbers. In \S\:6, we give our main result on enumerating maximal Lagrangian subbundles (Corollary \ref{counting_formula}). At the end, the numbers are explicitly computed for ranks two and four (Corollary \ref{examples}).

\subsection*{Acknowledgements} The first and second authors would like to thank Oslo Metropolitan University for hospitality during a stay in T\o nsberg in June 2018, when this work was completed. The first and second authors were supported by Basic Science Research Programs through the National Research Foundation of Korea (NRF) funded by the Ministry of Education (NRF-2016R1A6A3A11930321 and NRF-2017R1D1A1B03034277 respectively). The third author sincerely thanks Konkuk University, Hanyang University and the Korea Institute of Advanced Study for financial support and hospitality in March 2017.

\subsection*{Conventions and notation}
Throughout this paper, we work over the field $\C$ of complex numbers. Unless otherwise stated, $C$ is a complex projective smooth curve of genus $g \ge 0$, and $W$ is an $L$-valued symplectic bundle of rank $2n \ge 2$, where $\ell := \deg(L)$. The fiber at $p \in C$ of a bundle $V \to C$ is denoted by $V_p$. We consider points of Quot schemes as subsheaves, and use the notation $[E \to W]$.

\section{Quantum cohomology of Lagrangian Grassmannians}

In this section, we record some known facts on quantum cohomology of Lagrangian Grassmannians.

\subsection{Notations} \label{CombinatoricNotation} Fix a positive integer $n$. A \textit{partition} $\lambda$ is a weakly decreasing sequence of nonnegative integers $\lambda = (\lambda_1 , \lambda_2 , \ldots , \lambda_n)$. The nonzero $\lambda_i$ are called the \textit{parts} of $\lambda$. The number of parts is called the \textit{length} of $\lambda$ and is denoted $l(\lambda)$. The sum $\sum_{i=1}^n \lambda_i$ is called the \textit{weight} of $\lambda$, and denoted $|\lambda|$.

Denote by $\Rn$ the set of all partitions $(\lambda_1 , \ldots , \lambda_n)$ such that $\lambda_1 \leq n$. A partition $\lambda$ is called \textit{strict} if $\lambda_1 > \cdots >\lambda_l$ and $\lambda_{l+1} = \cdots = \lambda_n = 0$, where $l = l(\lambda)$. Let $\Dn$ be the set of all strict partitions $(\lambda_1, \ldots ,\lambda_n) \in \Rn$ such that $\lambda_1 \leq n$. We usually write $(\lambda_1, \ldots ,\lambda_l)$ for a (strict) partition $\lambda = (\lambda_1, \ldots , \lambda_l, 0, \ldots , 0)$ of length $l$, if no confusion should arise. For $\lambda \in \Dn$, let $\lambda^\prime$ be the \textit{dual} partition of $\lambda$, whose parts complement the parts of $\lambda$ in the set $\{ 1 , 2, \ldots , n \}$. Set $\rho_n = (n, n-1 , \ldots , 1) \in \Dn$. 
 
Later, we shall also use the following notations to state the Vafa--Intriligator-type formula in \S\:\ref{GW_for_LG(n)}. For $n = 2m+1$, set
\[
\mathcal{T}_n \ := \ \{ J = (j_1, \ldots , j_n) \in \mathbb{Z}^n \: | \: -m \leq j_1 < \dots < j_n \leq 3m+1 \} ,
\]
and for $n = 2m$, set
\[
\mathcal{T}_n \ := \ \left\{ J = (j_1 , \ldots , j_n) \in \left( \mathbb{Z} + \frac{1}{2} \right)^n \: {\big|} \: -m +\frac{1}{2} \leq j_1 < \dots < j_n \leq 3m- \frac{1}{2} \right\} .
\]
For $J = (j_1, \ldots , j_n) \in \mathcal{T}_n$ and $\zeta := \exp \left({\frac{\pi \sqrt{-1}}{n}}\right)$, we write $\zeta^J := (\zeta^{j_1}, \ldots ,\zeta^{j_n})$. Define a subset $\In$ of $\mathcal{T}_n$ by
\[ \In \ := \ \left\{ J = (j_1, \ldots , j_n) \in \mathcal{T}_n \hspace{0.05in}| \: \zeta^{j_k} \ne -\zeta^{j_l}  \text{ for } k \neq l \right\} . \]
Note that $\prod_k \zeta^{j_k} = \pm 1$ for $J = (j_1 , \ldots , j_n) \in \In$. We put
$$ \In^e \ := \ \left\{ J \in \In \hspace{0.05in}\big|\hspace{0.05in} {\prod}_k \zeta^{j_k} = 1 \right\} . $$

\subsection{Symmetric polynomials}
Let $X = (x_1, \ldots , x_n)$ be an $n$-tuple of variables. For $i=1, \ldots , n$, let $H_i (X)$ (resp. $E_i (X)$) be the $i$-th complete (resp., elementary) symmetric function in $X$. Then for any partition $\lambda$, the Schur polynomial $S_{\lambda}(X)$ is defined by
$$ S_{\lambda}(X) \ := \ \det \left[ H_{\lambda_i+j-i}(X) \right]_{1 \leq i , j \leq n } $$
where $H_0 (X) = 1$, and $H_k (X) = 0$ for $k < 0$.

The $\Qt$-polynomials of Pragacz and Ratajski \cite{PrRa} are indexed by the elements of $\Rn$. For $i \geq j$, define
\begin{displaymath}
\Qt_{i,j}(X) \ = \ E_i (X) E_j (X) + 2 \sum_{k=1}^{j}(-1)^k E_{i+k} (X) E_{j-k} (X) .
\end{displaymath}
For any partition $\lambda$, not necessarily strict, and for $r = 2 \left\lfloor (l(\lambda)+1)/2 \right\rfloor$, let $B_\lambda$ be the $r \times r$ skewsymmetric matrix whose $(i, j)$-th entry is given by $\Qt_{\lambda_i , \lambda_j} (X)$ for $i < j$. The $\Qt$-polynomial associated to $\lambda$ is defined by
\[ \displaystyle \Qt_{\lambda}(X) \ = \ \textrm{Pfaff} (B_{\lambda}) . \]
Note that from the definition of $\Qt_{\lambda}(X)$, for $\lambda = (k)$ with $0 \leq k \leq n$ we have $\Qt_{(k)}(X) = E_k (X)$. We often write $\Qt_k(X)$ for $\Qt_{(k)}(X)$.

\subsection{Degeneracy loci of type C}

Let $W$ be a vector bundle of rank $2n$ over a scheme $Z$, equipped with a symplectic form $\omega \colon W \otimes W \rightarrow \cO_Z$. Let $E$ be a vector bundle of rank $n$. Fix a homomorphism of vector bundles $\psi \colon E \rightarrow W$ with isotropic image; equivalently, such that the composite $E \rightarrow W \rightarrow W^\vee \xrightarrow{\psi^t} E^\vee$ is zero, where $W \rightarrow W^\vee$ is the isomorphism induced by the symplectic form $\omega$. Assume that $W$ admits a complete flag of isotropic subbundles
$$ H_\bullet : \ 0 \ = \ H_0 \ \subset \ H_1 \ \subset \ \cdots \ \subset \ H_n , $$
where $\rank ( H_k ) = k$. For any subbundle $G \subset W$, set
\[ G^\perp \ := \ \{ w \in W \: | \: \omega ( w \otimes v ) = 0 \text{ for all } v \in G \} , \]
the orthogonal complement of $G$ with respect to the symplectic form.

\begin{definition} \label{degeneracy-first} The degeneracy locus of type $C$ associated to a strict partition $\lambda \in \Dn$ is defined as
\begin{multline*} Z_\lambda ( H_\bullet ) \ := \ \left\{ z \in Z \hspace{0.04in} | \: \rank \left( E \ \to W/H_{n+1-\lambda_i}^\perp \right)_z \ \leq \ n + 1 - i - \lambda_i \right. \\ \left. \hbox{for each $i$} \right\}. \end{multline*}
\end{definition}

\noindent Note that $(W / H_{n+1-\lambda_i}^\perp )_z \cong \left( {H_{n+1-\lambda_i}^\vee} \right)_z$.

Degeneracy loci of type A are defined analogously, and their classes can be expressed in terms of the Chern classes of the vector bundles involved (see \cite{FuPr}). For type C, we have a similar expression when $\psi$ is everywhere injective. For $F$ a bundle of rank $n$ and $\lambda$ a partition, the class $\Qt_\lambda (F)$ is defined as $\Qt_\lambda(X)$ with the variable $x_i$ specialized to the $i$th Chern root of $F$. Recall that if $\lambda = (k)$ where $1 \leq k \leq n$, then $\Qt_{\lambda}(X) = E_k(X)$. This implies that $\Qt_{\lambda}(F) = c_k(F)$.

The following is a special case of \cite[Corollary 4]{KT2002}.

\begin{prop}\label{Prop:Lagrangian-Loci} Suppose that $Z$ is Cohen--Macaulay, and that the subbundles $H_1 , \ldots , H_n$ are trivial over $Z$. Assume that $Z_\lambda (H_\bullet)$ is of pure codimension $|\lambda|$. If $\psi \colon E \to W$ defines a Lagrangian subbundle, then in the Chow group $\Chow^{|\lambda|}(Z)$, we have $[Z_{\lambda}( {H}_\bullet)] \ = \ \Qt_\lambda ( {E}^\vee )$. \end{prop}

\begin{proof} See \cite[Corollary 4]{KT2002} and the discussion on \cite[p.\ 1718]{KT2002}. \end{proof}

We remark that the condition that $\psi \colon E \to V$ be a vector bundle injection is necessary in general. A counterexample is described in \cite[\S \: 4.5]{KT2002} in a case where $\psi$ is not everywhere injective.

\subsection{Cohomology of Lagrangian Grassmannians}

Let $V$ be a vector space of dimension $2n$ equipped with a symplectic form $\omega \colon V \otimes V \rightarrow \C$. Let $\LG(V)$ be the Lagrangian Grassmannian parametrizing Lagrangian subspaces in $V$. Over $\LG(V)$, there is a universal exact sequence of bundles
\[ 0 \ \rightarrow \ {\mathbf E} \ \rightarrow \ {\mathbf V} \ \rightarrow \ {\mathbf E}^\vee \ \rightarrow \ 0 , \]
where ${\mathbf V} = \LG(V) \times V$. Clearly ${\mathbf V}$ admits a symplectic form induced from $V$, and the subbundle ${\mathbf E} \subset {\mathbf V}$ is isotropic. Let
$$ H_{\bullet} : \ H_1 \ \subset \ H_2 \ \subset \ \cdots \ \subset \ H_{n-1} \ \subset \ H_n $$
be a complete isotropic flag in $V$. This induces a complete flag of isotropic subbundles
$$ {\mathbf H}_{\bullet} : \ {\mathbf H}_1 \ \subset \ {\mathbf H}_2 \ \subset \ \cdots \ \subset \ {\mathbf H}_{n-1} \ \subset \ {\mathbf H}_n , $$
in ${\mathbf V}$, where ${\mathbf H}_k := \LG(V) \times H_k$. Then for strict partitions $\lambda \in \Dn$, the degeneracy loci
$Z_{\lambda}( {\mathbf H}_{\bullet})$ are called  \textit{Schubert varieties}. By Proposition \ref{Prop:Lagrangian-Loci}, we obtain
\begin{equation} \label{trivialcase}
\left[ Z_{\lambda}( {\mathbf H}_{\bullet}) \right] \ = \ \Qt_{\lambda}( {\mathbf E}^\vee ) .
\end{equation}
It is well known that the classes $\{ \sigma_\lambda := [Z_{\lambda}( {\mathbf H}_{\bullet})] \:|\: \lambda \in \Dn \}$ form a basis of the Chow group of $\LG(V)$. For $1 \le k \le n$, we have the length $1$ partition $(k)$. We write $\sigma_k$ for the \textit{special Schubert class} $\sigma_{(k)}\in \mathrm{CH}^k(\LG(V))$. 

\subsection{A Vafa--Intriligator-type formula} \label{GW_for_LG(n)}

Fix a symplectic vector space $V = \C^{2n}$, and write $\lg$ for $\LG(\C^{2n}).$ In this subsection, we state a Vafa--Intriligator-type formula for $\lg$, which computes the Gromov--Witten invariants. We begin by defining these invariants.

The \textit{degree} of a morphism $f \colon C \rightarrow \lg$ is defined as the intersection number
\[ \int_{[\lg]} f_*[C] \cdot \sigma_1 . \]
Such an $f$ defines a Lagrangian subbundle $E_f$ of the trivial symplectic bundle $\Oc^{\oplus 2n}$, and $\deg(E_f) = -\deg(f)$. The Gromov--Witten invariant is informally defined as follows. For the precise definition, see \cite{RuTi}.

\begin{definition} \label{GW-for-LG} Let $p_1, \ldots , p_m$ be distinct points of $C$. Let ${\lambda^1}, \ldots , {\lambda^m} \in \Dn$ be strict partitions. Fix $d \in \Z$. We define the \emph{Gromov--Witten invariant} $\langle \sigma_{\lambda^1}, \ldots ,\sigma_{\lambda^m} \rangle_{C,d}$ as follows. If
\begin{equation} \label{dimcondition}
 \sum_{j=1}^m | \lambda^j | \ = \ \frac{1}{2}n(n+1)(1-g) + d(n+1) , 
 \end{equation}
then $\langle \sigma_{\lambda^1}, \ldots , \sigma_{\lambda^m} \rangle_{C, d}$ is the number of morphisms $f \colon C \rightarrow \lg$ of degree $d$, such that for each $i$, we have $f(p_i) \in Z_{\lambda^i}(\gamma_i \cdot \mathbf{H}_{\bullet})$ for a general choice of symplectic transformation $\gamma_i \in \Sp_{2n} (\C)$.

If (\ref{dimcondition}) does not hold, we define $\langle \sigma_{\lambda^1}, \ldots , \sigma_{\lambda^m}\rangle_{C,d}$ to be zero. \end{definition}

Now it is well known (see \cite[p. 262]{RuTi}) that the Gromov--Witten invariant is independent of the points $p_i$ and the curve $C$, depending only on the genus $g$. Thus we write $\langle \sigma_{\lambda^1},\ldots , \sigma_{\lambda^m}\rangle_{g, d}$ for $\langle \sigma_{\lambda^1},\ldots , \sigma_{\lambda^m}\rangle_{C, d}$.

\bigskip

The (small) quantum cohomology ring of $\lg$ is defined via the genus zero three-point Gromov--Witten invariants \cite{KT}. Let $q$ be a formal variable of degree $n+1$. The ring $qH^*(\lg,\Z)$ is isomorphic as a $\Z[q]$-module to $H^*(\lg, \Z) \otimes_{\Z} \Z[q]$. The multiplication in $qH^*(\lg,\Z)$ is given by the formula
\[
\sigma_\lambda \cdot \sigma_\mu \ = \ \sum_{d \ge 0 } \sum_\nu \left\langle \sigma_\lambda , \sigma_\mu, \sigma_{\nu'} \right\rangle_{0,d} \sigma_\nu \: q^d,
\]
where $\nu$ ranges over all strict partitions with $|\nu| = |\lambda| + |\mu| - (n+1)d$. Note that the specialization of the (complexified) quantum cohomology ring at $q=1$ is given by
\[
qH^*(\lg, \C)_{q=1} \ := \ qH^*(\lg, \C) \otimes \C[q]/(q-1).
\]
As a complex vector space, this is isomorphic to $H^*( \lg , \C )$.

Now we are ready to give a Vafa--Intriligator-type formula for $\lg$ for an arbitrary genus $g$.

\begin{prop} \label{vafa-Int-g}
Let $C$ be a curve of genus $g$ with $m$ marked points. For strict partitions $\lambda^1, \lambda^2, \dots, \lambda^m \in \Dn$ and $d \ge 0$, the genus $g$ Gromov--Witten invariant for $\lg$ is computed as
\[
\left\langle \sigma_{\lambda^1}, \sigma_{\lambda^2}, \ldots , \sigma_{{\lambda}^m} \right\rangle_{g,d} \ := \ {2^{n(g-1)-d}} \sum_{J\in \mathcal{I}_{n+1}^e} {S_{\rho_{n}}(\zeta^J)}^{g-1} {\Qt_{\lambda^1}(\zeta^J) \cdots \Qt_{\lambda^m}(\zeta^J)}
\]
 whenever $\sum_{i=1}^m |\lambda^i| = \frac{n(n+1)}{2}(1-g)+(n+1)d$, and zero otherwise. 
\end{prop}
\begin{proof}
For $g = 0$, the formula was given in \cite{Che1}. For an arbitrary $g$, we have the formula from \cite[p.\:1263]{CMP}:
\be \label{genus-g-VI-formula}
\left\langle \sigma_{\lambda^1} , \sigma_{\lambda^2}, \ldots , \sigma_{{\lambda}^m} \right\rangle_{g,d} \ = \ \mathrm{tr} \left( [ \mathcal{E}^{g-1}\sigma_{\lambda^1} \cdots \sigma_{\lambda^m}] \right) ,
\ee
where $\mathcal{E}$ is the quantum Euler class (cf.\ \cite{Ab}) of $\lg$ in $qH^*(\lg,\C)_{q=1}$, and $[\sigma]$ denotes the quantum multiplication operator on $qH^*(\lg,\C)_{q=1}$ determined by $\sigma$. Then the formula follows from \cite[Theorem 6.6]{Che2} where the eigenvalues of $[\sigma]$ were computed for an arbitrary $\sigma \in qH^*(\lg,\C)_{q=1}$. \end{proof}

\section{Lagrangian Quot Schemes}

Let $\Mor^d(C,\lg)$ be the space of morphisms of degree $d$ from $C$ to $\lg$. Informally, Gromov--Witten invariants of $\lg$ might be thought of as intersection numbers on $\Mor^d(C,\lg)$. However, it is necessary to compactify $\Mor^d(C,\lg)$ in order to develop an intersection theory. An alternative compactification to Kontsevich's moduli space of stable maps is the Lagrangian Quot scheme, which is practical for computing intersection numbers. In fact, Kresch and Tamvakis \cite{KT} used a Lagrangian Quot scheme for $W = \mathcal{O}_{\mathbb{P}^1}^{\oplus 2n}$ to compute the quantum cohomology ring of $\lg$. This may indicate that Lagrangian Quot schemes are important moduli spaces whose intersection theory is of interest. In this section, we describe the Lagrangian Quot schemes of a symplectic bundle over a curve of any genus.

\subsection{Definition and notation}

Let $C$ be a smooth projective curve of genus $g$, and $W$ a vector bundle on $C$. Let $\Quot_{k,e}(W)$ be Grothendieck's Quot scheme parametrizing subsheaves $[E \to W]$ of rank $k$ and degree $e$, or equivalently quotients $[W \to F]$ where $F$ is coherent of degree $\deg(W) - e$ and rank $r-k$. Let $\Quot_{k,e}^\circ(W)$ be the open sublocus
\[ \left\{ [ \psi \colon E \to W ] \in \Quot_{k,e}(W) \: | \: \psi \hbox{ is a vector bundle injection} \right\} . \]
Recall that $\Quot_{k,e}(W)$ is a projective variety, possibly having other components than the closure of $\Quot_{k,e}^\circ(W)$. If $\pi_C \colon \Quot_{k,e}(W) \times C \rightarrow C$ is the projection, then on $\Quot_{k,e}(W) \times C$ we have the universal exact sequence of sheaves
$$ 0 \ \rightarrow \ \mathcal{E} \ \rightarrow \  \pi_C^*W \ \rightarrow \ \mathcal{Q}\rightarrow \ 0 . $$

Suppose now that $\rank (W) = 2n$ and $W$ is equipped with a symplectic form $\omega \colon W \otimes W \rightarrow L$, where $L$ is a line bundle of degree $\ell$. As $\omega$ induces an isomorphism $W \cong W^\vee \otimes L$, in particular $\deg (W) = n\ell$.

The {\it Lagrangian Quot scheme} $\LQ$ is the subscheme of $\Quot_{n,e}(W)$ consisting of Lagrangian subsheaves. To see that $\LQ$ is a closed subscheme of $\Quot_{n,e}(W)$, consider the map
\[
\sigma \colon \Quot_{n,e}(W) \ \longrightarrow \ H^0(C, \wedge^2 ({\mathcal E}^\vee) \otimes L)
\]
sending $[ j \colon E \to W]$ to $\omega \circ \wedge^2 j \colon \wedge^2 E \to L$. This $\sigma$ defines a section of the sheaf $\pi_* (\wedge^2 (\mathcal{E}^\vee) \otimes \pi^*L)$, where $\pi \colon \Quot_{k,e}(W) \times C \rightarrow \Quot_{k,e}(W)$ is the projection. The subscheme $\LQ$ is nothing but the zero locus of $\sigma$. (For another argument, see \cite[Lemma 2.2]{CCH}.)

Hence $\LQ$ is a compactification of the quasiprojective scheme $\LQo$ of Lagrangian subbundles, possibly having components in addition to the closure of $\LQo$. For the trivial symplectic bundle $W = \Oc^{\oplus 2n}$ and $e \le 0$, the subscheme $\LQo$ coincides with the space $\Mor^{-e}(C, \lg)$ of morphisms of degree $-e$.

\subsection{Property \texorpdfstring{$\cP$ on $\LQ$}{P}} \label{P01}
In this subsection, we discuss further the property $\cP$ which was defined in \S\:1. To give a motivating example of an $\LQ$ having property $\cP$, we use the notion of very stability as studied in \cite{BR}. A symplectic bundle $W \cong W^\vee \otimes L$ is called \textit{very stable} if the bundle $K_C L^{-1} \otimes \Sym^2 W$ has no nonzero nilpotent sections. The following is proven similarly to \cite[Lemma 3.3]{LN}.

\begin{lemma} \label{very-stable} Let $W$ be a very stable symplectic bundle. Then we have $H^1 ( C, L \otimes \Sym^2 E^\vee ) = 0$ for every Lagrangian subsheaf $E \subset W$. \end{lemma}

By \cite[Proposition 2.4]{CCH}, the Zariski tangent space of $\LQ$ at a point $[ E \to W ] \in \LQo$ is $H^0 ( C, L \otimes \Sym^2 E^\vee)$. Hence the expected dimension of $\LQ$ is
\begin{equation} \label{Dnel}
 \chi (C, L \otimes \Sym^2 E^\vee) \ = \ -(n+1)e - \frac{n(n+1)}{2}(g-1-\ell) \ = \ D ( n, e, \ell ) . 
 \end{equation}

\begin{prop} \label{Prop-zero}
Let $L$ be a line bundle of degree $\ell$ and $W$ an $L$-valued symplectic bundle of rank $2n$ which is general in moduli. Set $e_0 := - \left\lceil \frac{1}{2}{n(g-1-\ell)} \right\rceil$.
\begin{enumerate}
\item A maximal Lagrangian subbundle of $W$ has degree $e_0$, and $LQ_{e_0}(W) = LQ_{e_0}^\circ(W)$.
\item If $n(g-1-\ell)$ is even, then $LQ_{e_0}(W)$ is a smooth scheme of dimension zero.
\end{enumerate}
\end{prop}

\begin{proof} The first statement in (1) follows from \cite[Theorem 1.4 and Remark 3.6]{CH2}. For the rest: As the Lagrangian subsheaves parametrized by $LQ_{e_0} (W)$ have maximal degree in $W$, every point of $\LQ$ corresponds to a Lagrangian subbundle, for otherwise, the subbundle generated by a subsheaf of degree $e_0$ would be a Lagrangian subbundle of higher degree. Hence in this case $\LQ = LQ_e^\circ(W)$.

For (2): By \cite[Proposition 2.4]{CCH}, if $H^1 (C, L \otimes \Sym^2 E^\vee) = 0$ then $\LQ$ is smooth of dimension $\chi (C, L \otimes \Sym^2 E^\vee) = 0$ at $[E \to W]$. By Lemma \ref{very-stable}, this holds for all $[E \to W]$ if $W$ is very stable; and by \cite{BR}, very stable bundles are dense in moduli. Statement (2) follows. \end{proof}

In particular, if $W$ is generic and $n(g-1-\ell) \equiv 0 \mod 2$, then $LQ_{e_0} (W)$ has property $\cP$. More generally, regarding property $\cP$, we cite the main result of \cite{CCH}.

\begin{prop} \label{prop:irreducible}
Let $W$ be a symplectic bundle of degree $n\ell$ over $C$. Then there exists an integer $e(W)$ such that if $e \leq e(W)$, then $\LQ$ is an irreducible and generically smooth variety of dimension $D ( n, e, \ell )$, 
 of which a general point corresponds to a Lagrangian subbundle. In particular, if $e \le e(W)$, then $\LQ$ has property $\cP$. \end{prop}

\subsection{Nonsaturated loci of Lagrangian Quot schemes} \label{NonSat}

Let $W$ be a symplectic bundle and $E \subset W$ a Lagrangian subsheaf. We denote by $\bE$ the saturation of $E$ in $W$. This is the sheaf of sections of the subbundle generated by $E$, or equivalently, the inverse image in $W$ of the torsion subsheaf of $W/E$.
 For fixed $e$ and for $r \ge 0$, we write
\[ \cB_r \ := \ \{ [E \to W] \in \LQ \: | \: \bE / E \hbox{ is a torsion sheaf of length } r \} . \]
This is a locally closed subscheme of $\LQ$. The following is clear from the definitions (compare with \cite[Theorem 1.4]{Be}).

\begin{lemma} \label{BrFibration} The association $E \mapsto \bE$ defines a surjective morphism
\[ f_r \colon \cB_r \ \to \ LQ^\circ_{e+r} (W) . \] If $F \subset W$ is a Lagrangian subbundle of degree $e+r$, then $f_r^{-1} ( F )$ is canonically identified with $\Quot^{0, r} (F)$. In particular, $\cB_r \to LQ_{e+r}^\circ ( W )$ is topologically a fiber bundle with irreducible fibers of dimension $nr$. \end{lemma}

\noindent Notice that $f_0 \colon \cB_0 \to \LQo$ is the identity map.

\section{Degeneracy loci for Lagrangian Quot schemes}

\subsection{Lagrangian degeneracy loci and Chern classes}

In Proposition \ref{Prop:Lagrangian-Loci}, we recalled that Lagrangian degeneracy loci $\left[ Z_\Lambda ( H_\bullet ) \right]$ can be expressed in terms of Chern classes when $\psi \colon E \hookrightarrow V$ is a vector bundle injection. Here we will see that for special Schubert classes, this condition on $\psi$ can be relaxed.

Let $W$ be an $L$-valued symplectic bundle, and set $\bbX := \LQ$. Write $\pi_C \colon \bbX \times C \to C$ for the projection. There is an exact sequence of sheaves over $\bbX \times C$ given by
\begin{equation} 0 \ \rightarrow \ \mathcal{E} \ \rightarrow \ \pi_C^* W \ \rightarrow \ \widetilde{\mathcal{E}} \end{equation} \label{univ}
where $\cE$ is the universal subsheaf and $\widetilde{\cE} = \cE^\vee \otimes \pi_C^*L$. For $p \in C$, denote by $\mathcal{E}(p)$, $\mathcal{W}(p)$ and $\widetilde{\mathcal{E}}(p)$ the restrictions to $\bbX \times\{ p \}$ of $\mathcal{E}$, $\pi_C^* W$ and $\widetilde{\mathcal{E}}$ respectively.

Identifying the fibers $\widetilde{\cE}(p) \cong \cE(p)^\vee$, we obtain
\be \label{exact-sequence-p} 0 \ \rightarrow \ \cE(p) \ \rightarrow \ \cW(p) \ \rightarrow \cE(p)^\vee . \ee
Note that $\cW(p) = \bbX \times W_p$ is a trivial symplectic bundle. Let
$$H_{\bullet} : \ H_1 \ \subset \ H_{2} \ \subset \ \cdots \ \subset \ H_{n-1} \ \subset \ H_n $$
be a complete flag of isotropic subspaces in $W_p$, and
\[ H_n \ = \ H_n^\perp \ \subset \ H_{n+1}^\perp \ \subset \ \cdots \ \subset \ H_2^\perp \ \subset \ H_1^\perp \]
the corresponding coisotropic flag of orthogonal complements. This induces a flag of trivial subbundles
\[ \mathcal{H}_{n-k+1}^\perp \ := \ \bbX \times H_{n-k+1}^\perp : \ 1 \le k \le n \]
of $\bbX \times W_p$. Following \cite{KT}, we will define Lagrangian degeneracy loci on $\bbX$. Each Lagrangian subsheaf map $\psi \colon E \to W$ induces a map $E_p \to W_p / H_{n+1-k}^\perp$ for each $k$. We adapt Definition \ref{degeneracy-first} to this case.

\begin{definition}\label{def: Lagrangian loci} For $p \in C$ and $\lambda \in \Dn$, define $\bbX_\lambda(H_{\bullet};p)$ as
\begin{multline*} \left\{ [ \psi \colon E \to W ] \in \bbX \hspace{0.05in} | \hspace{0.05in} \rank \left( \cE(p) \rightarrow \cW(p)/ \mathcal{H}^{\perp}_{n+1-\lambda_i} \right)_\psi \ \leq \ n+1-i-\lambda_i , \right. \\ \left. 1\le i \le l(\lambda) \right\} .  \end{multline*} 
\end{definition}

\begin{remark} \label{independent_p} If $\lambda = (k)$, then $\bbX_{(k)}(H_{\bullet} ; p)$ is determined by the single isotropic subspace $H_{n+1-k}$ of $W_p$. Also, for $\rho_n = (n, n-1, \dots, 1)$, we have
$$ \bbX_{\rho_n} ( H_{\bullet} ; p ) \ = \ \left\{ [ \psi \colon E \to W ] \in \bbX \: | \: \psi \left( \cE(p) \right) \subseteq H_n \right\} , $$
which depends only on $H_n$. Henceforth we shall denote $\bbX_{(k)}( H_{\bullet} ; p )$ and $\bbX_{\rho_n}(H_{\bullet};p)$ by $\bbX_k ( H_{n+1-k} ; p) $ and $\bbX_{\rho_n}(H_n ; p)$ respectively. \end{remark}

\begin{lemma}\label{Lemma:Chern Class} Assume that $\bbX = \LQ$ has property $\cP$. Fix $\lambda\in \Dn$. Suppose that $\bbX_{\lambda}( H_\bullet ; p)$ satisfies the following.
  \begin{enumerate} \item Every irreducible component of $\bbX_{\lambda}( H_\bullet ; p )$ has codimension $|\lambda|$.
  \item \label{inquality} $\dim \hspace{0.03in}\bbX_{\lambda}( H_\bullet ; p ) \cap (\bbX\setminus \bbX^\circ) < \dim \hspace{0.03in}\bbX_{\lambda}(H_\bullet ; p)$.
  \end{enumerate}
Then we have $\left[ \bbX_{\lambda}(H_\bullet ; p) \right] = \Qt_{\lambda}(\mathcal{E}^\vee(p))\cap [\bbX]$. \end{lemma}

\begin{proof}
To ease notation, write $\bbX_{\lambda} := \bbX_{\lambda}( p ; H_\bullet )$ and $\bbX_{\lambda}^\circ := \bbX_{\lambda}^\circ( p ; H_\bullet )$ for the duration of this proof. Set also $m := D ( n, e, \ell ) - |\lambda|$, where $D(n, e, \ell)$ is as given in (\ref{Dnel}). Note that property $\cP$ and conditions (1) and (2) imply that $\bbX_\lambda$ is the closure of $\bbX_{\lambda}^\circ$ in $\bbX$, and that $\dim ( \bbX_\lambda ) = \dim ( \bbX_\lambda^\circ ) = m$.
 
Firstly, by Proposition \ref{Prop:Lagrangian-Loci}, we have
\[ [ Z_{\lambda}( H_\bullet ) ] \ = \ \Qt_{\lambda}(\mathbf{E}^\vee) \cap [\LG(W_p)] . \]
The subvariety $\bbX^\circ_{\lambda}$ is the preimage of $Z_{\lambda}(H_\bullet)$ under $\ev_p \colon \LQo \rightarrow \LG(W_p)$, and $\mathcal{E}^\vee(p) = \ev_p^*(\mathbf{E}^\vee)$. Hence we have equality
\begin{equation} [\bbX^\circ_{\lambda} ] \ = \ \Qt_{\lambda}(\mathcal{E}^\vee(p)) \cap [\bbX^\circ] \ \in \ \Chow_{m}(\bbX^\circ) . \label{QtXo} \end{equation}
Equivalently, the image of $[\bbX_{\lambda}^\circ]$ is $\Qt_{\lambda}(\mathcal{E}^\vee(p)) \cap [\bbX^\circ]$ under the natural homomorphism $\Chow_m ( \bbX_{\lambda}^\circ ) \rightarrow \Chow_m ( \bbX^\circ )$. (By convention, the image of $[\bbX_{\lambda}^\circ]$ is also denoted by $[\bbX_{\lambda}^\circ]$.) By (1) and (2), this in turn implies that the class $\Qt_{\lambda}(\mathcal{E}^\vee(p)) \cap [\bbX]$ lies in the image of the homomorphism $\Chow_m ( \bbX_{\lambda} ) \rightarrow \Chow_m ( \bbX )$.

Now let us show that $\Qt_{\lambda}(\mathcal{E}^\vee(p)) \cap [\bbX]$ is the image of $[\bbX_{\lambda}]$ under the homomorphism $\Chow_m ( \bbX_{\lambda} ) \rightarrow \Chow_m ( \bbX )$; that is,
\begin{equation} [\bbX_{\lambda} ] \ = \  \Qt_{\lambda}(\mathcal{E}^\vee(p)) \cap [\bbX] \ \in \ \mathrm{CH}_{m}(\bbX). \end{equation}

By \cite[p.\ 21]{Fu}, there is a commutative diagram
$$ \xymatrix{ \Chow_m ( \bbX_{\lambda} \setminus \bbX_{\lambda}^\circ ) \ar[d]\ar[r]^-{i_*}& \Chow_m ( \bbX_{\lambda} ) \ar[d]\ar[r]^{j^*}& \Chow_m ( \bbX_\lambda^\circ )\ar[d]\ar[r]&0\\
 \Chow_m ( \bbX\setminus \bbX^\circ )\ar[r]^-{\tilde{i}_*}&\Chow_m ( \bbX)\ar[r]^{\tilde{j}^*}& \Chow_m ( \bbX^\circ )\ar[r]&0. }
$$
Here $i, j, \tilde{i}, \tilde{j}$ are embeddings. The group $\Chow_m ( \bbX_{\lambda} \setminus \bbX_{\lambda}^\circ )$ is zero by the dimension assumption (2), so $j^*$ is an isomorphism. Since $j^*[\bbX_{\lambda}] = [\bbX_{\lambda}^\circ]$ and
$$ \tilde{j}^* \left( \Qt_{\lambda}(\mathcal{E}^\vee(p)) \cap [\bbX] \right) \ = \ \Qt_{\lambda} ( \mathcal{E}^\vee(p) ) \cap [ \bbX^\circ ] , $$
by a diagram chase we see that the image of $[\bbX_{\lambda}]$ under the homomorphism $\Chow_m(\bbX_{\lambda})\rightarrow \Chow_m(\bbX))$
is $\Qt_{\lambda}(\mathcal{E}^\vee(p)) \cap [\bbX].$ This proves the lemma. \end{proof}

\begin{prop} \label{prop:Chern} Suppose $\bbX = \LQ$ has property $\cP$. For $1 \le k \le n$, let $H_{n+1-k}$ be an isotropic subspace of $W_p$ for a point $p\in C.$ Then in $\Chow_*(\bbX)$ we have the equality
$$ \left[ \bbX_k ( H_{n+1-k}; p ) \right] \ = \ c_k ( \mathcal{E}^\vee(p) ) \cap [\bbX] . $$ \end{prop}

\begin{proof}  For general $\gamma \in \Sp ( W_p )$, we have $[ \bbX_k ( H_{n+1-k}; p ) ] = [ \bbX_k (\gamma \cdot H_{n+1-k};p ) ]$. Thus the equality would follow if we show that $\bbX_k (\gamma \cdot H_{n+1-k};p )$ satisfies the conditions of Lemma \ref{Lemma:Chern Class}. 
 As each irreducible component of has codimension at most $k$ in $\LQ$, to check these conditions 
 it is enough to show that the intersection $\bbX_k( \gamma \cdot H_{n + 1 - k} ; p) \cap \cB_r$ has codimension at least $k$ in $\cB_r$ for each $r \ge 0$. We adapt the approach of \cite[Theorem 1.4]{Be}.

To ease notation, write $H := H_{n+1-k}$. For a fixed $r \ge 0$, we define the degeneracy locus
\[ \bbY_k^\circ ( H ; p ) \ := \ \left\{ [ F \to W ] \in LQ_{e+r}^\circ ( W ) \: | \: \rank ( F_p \to W_p / H^\perp ) \ \le \ n - k \right\} . \]
This is the preimage by $\ev_p \colon LQ_{e+r}^\circ \to \LG ( W_p )$ of the degeneracy locus
\[ Z_{k} ( H ) \ = \ \{ \Lambda \in \LG \left( W_p \right) \: | \: \dim ( \Lambda \cap H^\perp ) \ \ge \ k \} \]
(cf.\ Definition \ref{degeneracy-first}), which has codimension $k$. Now ${\ev_p}|_{LQ_{e+r}^\circ (W)}$ is a morphism, so by Kleiman \cite[Theorem 2]{Kle}, for a general $\gamma \in \Sp( W_p )$, the locus $\bbY_k^\circ ( \gamma \cdot H ; p)$ is of codimension $k$ in $LQ_{e+r}^\circ (W)$ unless it is empty.

Consider now the set
\[ \{ E \in \bbX_k ( \gamma \cdot H ; p ) \cap \cB_r \: | \: \rank ( \bE_p \to W_p / H^\perp ) \ \le \ n - k \} . \]
This is precisely $f_r^{-1} \left( \bbY_k^\circ ( \gamma \cdot H ; p ) \right)$, where $f_r \colon \cB_r \to LQ_{e+r}^\circ (W)$ is as defined in {\S}\:\ref{NonSat}. By Lemma \ref{BrFibration} and the last paragraph, $f_r^{-1} \left( \bbY_k^\circ ( \gamma \cdot H ; p )\right)$ has codimension $k$ in $\cB_r$.

(For the remainder of the proof, we do not use the assumption that $\gamma$ is general.) It remains to treat the situation where $\bE_p \to W_p / H^\perp$ is surjective. In this case, $E$ can belong to $\bbX_k ( H ; p )$ only if $E$ fails to be saturated at $p$ (in particular, $r \ge 1$). Since
\[ \mathrm{length} ( \bE_p / E_p ) \ \le \ \deg ( \bE / E ) \ = \ r , \]
we have $\rank ( \psi_p \colon E_p \to \bE_p) \ge n - r$. Thus, for each $l$ satisfying
\[ \max\{ 0, n - r \} \ \le \ l \ \le \ n - 1 , \]
we consider the set $\{ E \in \bbX_k ( H ; p ) \cap \cB_r \: | \: \rank ( E_p \to \bE_p ) \ \le \ l \}$.

Now for a fixed $F \in LQ_{e+r}^\circ (W)$, we claim that the locus
\begin{equation} \{ [ E \to F ] \in \Quot^{0, r} ( F ) \: | \: \rank ( E_p \to F_p ) \ = \ l \} \label{rankl} \end{equation}
is of codimension $(n-l)^2$ in $\Quot^{0, r} ( F )$. For; the image of $E_p \to F_p$ is determined by the choice of a point in $\Gr ( l, F_p )$, which has dimension $l(n-l)$. The remaining torsion of $F/E$ has degree $r-(n-l)$. This is determined by the choice of a point in an open subset of $\Quot^{0, r-n+l} (F)$, which has dimension $(r-n+l)n$. Thus, as desired, the dimension of (\ref{rankl}) is
\[ l(n-l) + (r-n+l)n 
\ = \ \dim \Quot^{0, r} (F) - (n-l)^2 . \]

Suppose firstly that $l \le n - k$. Then $E_p \to W_p / H^\perp$ cannot be surjective. In this case $k \le n-l \le (n-l)^2$, so (\ref{rankl}) has codimension at least $k$ in $\Quot^{0, r} (F) = f_r^{-1} ( F )$. By Lemma \ref{BrFibration}, for $l \le n - k$ the union of the loci (\ref{rankl}) as $F$ varies in $LQ_{e+r}^\circ$ is of codimension at least $k$ in $\cB_r$.

On the other hand, suppose $l \ge n - k + 1$. Noting that
\[ \Ker \left( \psi ( E_p ) \to W_p / H^\perp \right) \ = \ \psi ( E_p ) \cap \left( H^\perp \cap F_p \right) , \]
we see that $E$ belongs to $\bbX_k ( H ; p )$ if and only if
\[ \dim \left( \psi ( E_p ) \cap H^\perp \cap F_p \right) \ \ge \ l - n + k \ = \ \dim \psi ( E_p ) - \dim W_p / H^\perp + 1 . \]
This is a Schubert condition on $\psi ( E_p ) \in \Gr ( l, F_p )$, of codimension $l - n + k$.
 Thus for $l \ge n - k + 1$, the locus
\[ \{ E \in f_r^{-1} (F) \cap \bbX_k ( H ; p ) \: | \: \rank ( \psi_p \colon E_p \to F_p ) \ = \ l \} \]
is of codimension
\[ (n-l)^2 + (k + l - n) \ = \ k + (n-l)(n-l-1) \ \ge \ k \]
in $f_r^{-1} (F) \cong \Quot^{0, r} ( F )$. (Notice that we have equality if $l = n-1$.) Letting $F$ vary in $LQ_{e+r}^\circ (W)$ as before, we see that the locus of $E \in \bbX_k ( p ; H )$ such that $\rank ( E_p \to \bE_p ) = l$ is of codimension at least $k$ in $\cB_r$. This completes the proof. \end{proof}

\noindent In Corollary \ref{cor-class}, we will give a similar result for $\lambda = \rho_n$.

\subsection{The Hecke transform}

In this subsection, given a vector bundle $V$ and a divisor $D$ on $C$, we denote $V \otimes \cO_C(D)$ by $V(D)$.

Let $W$ be a bundle with symplectic form $\omega \colon W \otimes W \rightarrow L$. Fix $p \in C$ and choose a subspace $\Lambda \subset W_p.$ Let $W_\Lambda$ be the Hecke transform of $W$, which is defined as the kernel of the composition map $W \rightarrow W_p \rightarrow W_p/\Lambda$. Then we have the exact sequence of sheaves
\be \label{hecke-transform}
0 \rightarrow \ W_\Lambda \ \rightarrow \ W \ {\rightarrow} \ W_p/\Lambda \ \rightarrow \ 0 .
\ee
By \cite[Proposition 2.2]{BG}, if $\Lambda$ is a Lagrangian subspace of $W_p$, then $W_\Lambda$ is bundle of degree $\deg(W)-n$ admitting the symplectic form
$$ \omega_\Lambda \colon W_\Lambda \otimes W_\Lambda \ \rightarrow \ L(-p) $$
and fitting into the commutative diagram
\be \label{commutative diagram1} \xymatrix{ W_\Lambda \otimes W_\Lambda \ar[r] \ar[d]_{\omega_\Lambda} & W \otimes W \ar[d]^{\omega} \\
L(-p) \ar[r] & L. } \ee

Dualizing (\ref{hecke-transform}), we obtain a sequence
\[
0 \to W^\vee \to (W_\Lambda)^\vee \to \C^n \otimes \cO_p \to 0 .
\]
Here $\cO_p$ is a skyscraper sheaf of length one supported at $p$. Using the isomorphisms $W \cong W^\vee \otimes L$ and $W_\Lambda \cong (W_\Lambda)^\vee \otimes L(-p)$, we obtain a sequence
\be \label{heckeseq}
0 \to W \to W^\Lambda  \to \C^n \otimes \cO_p \to 0,
\ee
where $W^\Lambda := W_\Lambda(p)$ is an $L(p)$-valued symplectic bundle. In this way, to each Lagrangian subspace $\Lambda \subset W_p$ we can associate a symplectic bundle $W^\Lambda$ fitting into (\ref{heckeseq}). Since $\Lambda^\vee \subset (W_\Lambda)^\vee_p$, we may regard $\Lambda^\vee$ as a Lagrangian subspace of $(W^\Lambda)_p$. If $E \subset W$ is a subsheaf, then $E$ can be viewed as a subsheaf of $W^\Lambda$ via the inclusion $W \subset W^\Lambda$. Furthermore, if $E\subset W$ is a Lagrangian subsheaf, so is $E \subset W^\Lambda$. Hence there is a well-defined morphism
$$ \Psi^{\Lambda} \colon LQ_e(W) \ \rightarrow \ LQ_e(W^\Lambda) . $$
One can check that $\Psi^{\Lambda}$ is an embedding. Furthermore, $\Psi^\Lambda ([E \to W])$ belongs to $LQ_e^\circ(W^\Lambda)$ if and only if $[E \to W] \in \LQo$ and $E_p \cap \Lambda = 0$ in $W_p$.

\begin{prop} \label{prop:relation} Fix $p \in C$ and a Lagrangian subspace $\Lambda \subset W_p$. Then the image of $\Psi^{\Lambda}$ coincides with the Lagrangian degeneracy locus $\bbX_{\rho_n}( \Lambda^\vee;p) \subseteq LQ_e ( W^\Lambda )$.
\end{prop}

\begin{proof} By definition, $[E \to W^\Lambda]$ belongs to $\bbX_{\rho_n}(\Lambda^\vee ; p)$ if and only if
 the map $E_p \rightarrow (W^\Lambda)_p$ factorizes via $\Lambda^\vee \subset (W^\Lambda)_p$. This is equivalent to $E \to W^\Lambda$ lifting to a degree $e$ Lagrangian subsheaf of $W$. \end{proof}

\begin{cor}\label{cor-class} Let $W$ be any symplectic bundle, and let $\Lambda$ be a general Lagrangian subspace of a fiber $p \in C$. Suppose that $\bbX := LQ_e(W^\Lambda)$ and $\LQ$ have property $\cP$. Then
\be \label{equality-class-gen} [\bbX_{\rho_n}(\Lambda^\vee)] \ = \ \Qt_{\rho_n}(\cE(p)^\vee)\cap [\bbX]  , \ee
in $\Chow_*\left( LQ_e ( W^\Lambda ) \right)$, where $\cE$ denotes the universal sheaf on $LQ_e ( W^\Lambda )$. \end{cor}

\begin{proof} To prove the corollary, we apply Lemma \ref{Lemma:Chern Class}. By hypothesis, both $\LQ$ and $LQ_e ( \tW )$ are of expected dimension. Hence by Proposition \ref{prop:relation}, every component of $\bbX_{\rho_n}(\Lambda^\vee ; p) = \Psi^\Lambda \left( \LQ \right)$ is of codimension
\[ D ( n, e, \ell + 1 ) - D (n, e, \ell ) 
\ = \ \frac{1}{2}n(n+1) \ = \ |\rho_n| . \]
This gives condition $(1)$ of Lemma \ref{Lemma:Chern Class}. Next, since $\Lambda$ is general, a general $[E \to W]$ in any component of $\LQ$ intersects $\Lambda$ in zero, and so $\Psi^\Lambda ( [ E \to W ] )$ is saturated for a generic $[E \to W]$. This implies that  the condition $(2)$ of Lemma \ref{Lemma:Chern Class} for $\bbX_{\rho_n}(\Lambda^\vee ; p)$ is satisfied. Thus, as desired, we have
$$ [\bbX_{\rho_n} ( \Lambda^\vee ; p) ] \  = \ \Qt_{\rho_n} ( \cE(p)^\vee ) \cap [ \bbX ] . \qedhere $$
\end{proof}

Now choose distinct points $q_1, \ldots, q_t \in C$ and Lagrangian subspaces $\Lambda_1, \ldots, \Lambda_t$ in $W_{q_1}, \ldots, W_{q_t}$ respectively. Let $\tW$ be the symplectic bundle obtained from a sequence of $t$ Hecke transforms associated to $\Lambda_1, \ldots, \Lambda_t$. Then $\tW$ fits into the sequence
\begin{equation} \label{Heckeseq}
0 \ \to \ W \ \to \ \tW \ \to \ \bigoplus_{j=1}^t\left( \C^n \otimes \cO_{q_j} \right) \ \to \ 0 .
\end{equation}
Then $\deg (\tW ) = \deg (W) + tn$, and as in the case above with $t = 1$, there is an embedding $\LQ \subset LQ_e(\tW)$.

\begin{lemma} \label{lemma-heckeP1} Let $W$ be any symplectic bundle. There is an integer $t$ such that if $\tW$ is the Hecke transform defined by a general choice of $t$ points $q_1, \ldots , q_t \in C$ and Lagrangian subspaces $\Lambda_i \subset W_{q_i}$, then the Lagrangian Quot scheme $LQ_e (\tW)$ has property $\cP$. \end{lemma}

\begin{proof} By Proposition \ref{prop:irreducible}, there exists $m_0 \ge 0$ such that $LQ_{e-mn} \left( W \right)$ has property $\cP$ for all $m \ge m_0$. Let $D$ be a reduced effective divisor of degree $m \ge m_0$. Then $W (D)$ is an $L(2D)$-valued symplectic bundle, and
\[ LQ_{e-mn} ( W ) \ \cong \ LQ_e ( W ( D ) ) \]
via the map $[ E \to W ] \mapsto [ E(D) \mapsto W ( D ) ]$.

Now $W ( D )$ is a symplectic Hecke transformation of $W$. Precisely, $W (D)$ is obtained from $W$ by transforming along $m$ pairs of complementary Lagrangian subspaces of $W$, one pair from each point of $\Supp (D)$. Clearly $W(D)$ can be deformed to the Hecke transform defined by a general choice of $2m$ Lagrangian subspaces of distinct fibers of $W$. Therefore, as property $\cP$ is open in families, a general Hecke transform $\widetilde{W}$ with $\deg \left( \widetilde{W} / W \right) = 2mn$ has property $\cP$.

Similarly, let $W^\Lambda$ be any Hecke transform of $W$ along a single Lagrangian subspace $\Lambda$. Applying the above argument to $W^\Lambda$, there exists $m_1 \ge 0$ such that if $m \ge m_1$ and $\widetilde{W^\Lambda}$ is a general Hecke transform of $W^\Lambda$ with $\deg \left( \widetilde{W^\Lambda} / W^\Lambda \right) = 2mn$, the scheme $LQ_e ( \widetilde{W^\Lambda} )$ has property $\cP$. But such a $\widetilde{W^\Lambda}$ is also a Hecke transform of $W$ along $2m+1$ Lagrangian subspaces (including $\Lambda$).

Thus, for $t \ge \max \{2m_0, 2m_1 + 1 \}$, if $\tW$ is the Hecke transform along a general choice of $t$ Lagrangian subspaces, then $LQ_e ( \tW )$ has property $\cP$. \end{proof}

\begin{cor}\label{coro:represent} Let $W$ be any symplectic bundle over $C$ and $\tW$ be the Hecke transform in (\ref{Heckeseq}). Assume $\LQ$ is not empty. Then, as subschemes of $LQ_e ( \tW )$, we have
$$ \LQ \ = \ \bigcap_{i=1}^t \bbX_{\rho_n}(\Lambda_i^\vee ; q_i) , $$
where we view $\Lambda_i^\vee$ as a Lagrangian subspace of $\tW_{q_i}$. Furthermore, if $\LQ$ has property $\cP$ and $t \gg 0$ and $\tW$ is general, then
\be \label{eq6} [\LQ] \ = \ \prod_{i=1}^t\Qt_{\rho_n}(\cE(q_i)^\vee) \cap [ LQ_e(\tW)] . \ee
\end{cor}

\begin{proof} The first equality follows by applying Proposition \ref{prop:relation} repeatedly. For the rest: By Lemma \ref{lemma-heckeP1}, we may assume that $LQ_e(\tW)$ has property $\cP$. Since $\LQ$ has property $\cP$, its image is of the expected codimension in $LQ_e ( \tW )$. Thus the equality follows from the first equality and Corollary \ref{cor-class}. \end{proof}

\section{Intersection theory on \texorpdfstring{$\LQ$}{LQe(W)}}

We shall now develop an intersection theory on $LQ_e(W)$. Let us give an outline of this section. We define  intersection numbers on $\LQ$ in two ways.

Firstly, as usual, an intersection number is defined as an integral of a cohomology class against the fundamental class $[\LQ]$. With this definition, for reasons which will become clear below, we shall restrict ourselves to those $\LQ$ having property $\cP$. (A motivating example is when $e = e_0$ and $\LQ$ is a finite number of smooth points, as discussed in {\S} \ref{P01}.)

Secondly, for an arbitrary $\LQ$, possibly containing oversized or generically nonreduced components, we embed $\LQ$ into a larger Lagrangian Quot scheme $LQ_e ( \tW )$ having property $\cP$, and then define intersection numbers on $\LQ$ via those on $LQ_e ( \tW )$. We show in Proposition \ref{definitions coincide} that these two definitions coincide when $\LQ$ has property $\cP$. Using this coincidence, we obtain relations among intersection numbers, and this in turn brings to us a relation to the Gromov--Witten invariants of $\lg.$

\subsection{Gromov--Witten invariants} \label{subsection:GW-invariants}

For $1 \le i \le n$, let $\alpha_i$ be a formal variable of weight $i$, and set $\alpha := (\alpha_1, \ldots ,\alpha_n)$. Recall that for a fixed $g$, we have defined
\[
D (n, e, \ell) \ := \ -(n+1)e-\frac{n(n+1)}{2}(g-\ell -1) ,
\]
the expected dimension of $\LQ$ for a symplectic bundle $W$ of rank $2n$ and degree $n \ell$ over a curve of genus $g$. The following should be compared with \cite[Definition 2.7]{Ho}.

\begin{definition} Let $W$ be a symplectic bundle of rank $2n$ and degree $w = n\ell$ over $C$. Suppose that $\LQ$ has property $\cP$. For a weighted homogeneous polynomial $P$, we define
$$
N_{C,e}^w (P ; W) \ := \ \int_{[LQ_e(W)]} P \left( c_1 (\cE(p)^\vee) , \dots ,  c_n (\cE(p)^\vee) \right)
$$
if $\deg \left( P(\alpha) \right) = D(n, e, \ell)$, and $N_{C,e}^w (P ; W) = 0$ otherwise. We call the number $N_{C,e}^w ( P ; W )$ a \emph{Gromov--Witten invariant} of the Lagrangian Grassmann bundle $\LG(W)$ over $C$. \end{definition}

\begin{remark}\label{remark:intersection} By \cite [Proposition 10.2] {Fu}, the number $N_{C,e}^w ( P ; W )$ does not depend on the chosen point $p \in C.$ More generally, we have
$$
N_{C,e}^w (P ; W) \ = \ \int_{[LQ_e(W)]} P \left( c_1 (\cE(p_1)^\vee) , \dots ,  c_n (\cE(p_n)^\vee) \right)
$$
for any $(p_1 , \ldots , p_n) \in C^n$. 
 Thus if $P ( \alpha ) = \prod _{i} \alpha_{k_i}$ is a monomial of weighted degree $D (n, e, \ell)$, then $N_{C,e}^w ( P ; W )$ enumerates the intersection $\bigcap_i \bbX_{k_i}(H^{(i)}_\bullet, p_i)$ for distinct general points $p_i$ and general flags $H_\bullet ^{(i)}$ in $W_{p_i}$. In particular, $N_{C,e}^w \left( \prod _{i} \alpha_{k_i} ; W \right)$ is a nonnegative integer.
\end{remark}

\noindent Let us show that $N^w_{C, e} ( P ; W )$ is a deformation invariant in families of Lagrangian Quot schemes with property $\cP$.

\begin{prop}\label{prop: not depend on b} Let $\cC \to B$ be a family of smooth projective curves over an irreducible curve $B$ and $\mathscr{L} \to \cC$ a line bundle of relative degree $\ell$. Let $\mathscr{W}$ be a vector bundle over $\mathcal C$ such that $\mathscr{W}|_b$ is an $\mathscr{L}|_b$-valued symplectic bundle over $\cC_b$ for each $b \in B$. Suppose that $LQ_e \left( \mathscr{W}|_b \right)$ has property $\cP$ for each $b \in B$. Then $N_{\cC_b,e}^w \left( P ; \mathscr{W}|_b \right)$ is independent of $b \in B$. \end{prop}

\begin{proof} The family $\mathscr{W} \to \cC$ gives rise to a family $\phi \colon LQ_e ( \mathscr{W} ) \to B$ of Lagrangian Quot schemes parametrized by $B$. Since by the hypothesis each fiber $LQ_e \left( \mathscr{W}|_b \right)$ is generically smooth of the expected dimension, as in the case of Quot schemes in \cite[Theorem 5.17]{Ko} the map $\phi$ is a local complete intersection morphism, and in particular flat. Thus the proposition follows from \cite[Lemma 1.6]{Be}. \end{proof}

\subsection{Definition of Gromov--Witten numbers on an arbitrary Lagrangian Quot scheme}

We shall now extend our definition of intersection number to arbitrary Lagrangian Quot schemes, not necessarily enjoying property $\cP$.

Let $W$ be any symplectic bundle of degree $w$ over $C$. Let $W^\Lambda$ be the symplectic bundle obtained by the Hecke transform (\ref{heckeseq}) associated to a general $\Lambda \in \LG(W_p)$.

\begin{lemma} \label{Prop:Hecke-GW} If both $\bbX = LQ_e (W^\Lambda)$ and $\LQ$ have property $\cP$, then
$$ N_{C,e}^w ( P(\alpha) ; W ) \ = \ N_{C,e}^{w+n} \left( \Qt_{\rho_n}(\alpha)P(\alpha) ; W^\Lambda \right) . $$
\end{lemma}

\begin{proof} We may assume that $\deg P(\alpha) = D (n, e, \ell )$, as both sides are zero otherwise. Recall from Corollary \ref{cor-class} that
\[ \left[ \LQ \right] \ = \ \Qt_{\rho_n}(\cE(p)^\vee) \cap \left[ \bbX \right] \]
in $\Chow_* \left( \bbX \right)$. Now the statement follows from the projection formula. \end{proof}

Motivated by the lemma, we make a definition.

\begin{definition} Let $W$ be a symplectic bundle of degree $w$, and suppose $\LQ$ is nonempty. For $t \gg 0$, let $\tW$ be a general symplectic Hecke transform of $W$ with $\deg (\tW / W) = tn$, so that $LQ_e(\tW)$ has property $\cP$ by Lemma \ref{lemma-heckeP1}. We define
$$ \widetilde{N}_{C,e}^w ( P(\alpha) ; W ) \: := \: N_{C, e}^{w + tn} \left( P(\alpha) \Qt_{\rho_n}^t ( \alpha ) ; \tW \right) . $$
\end{definition}

\begin{lemma} \label{lemma: Ntilde} The number $\tN_{C,e}^w ( P(\alpha) ; W )$ is well-defined and depends only on $g$, $e$ and $w$ once the polynomial $P(\alpha)$ is specified. More precisely,
\bn
\item It does not depend on the chosen Hecke transform $\tW$.
\item Let $\mathscr{W} \to \cC \to B$ be a family of symplectic bundles parametrized by a connected curve $B$, such that $LQ_e \left( \mathscr{W}|_b \right)$ is nonempty for all $b \in B$. Then $\tN_{\cC_b , e}^w \left( P(\alpha); \mathscr{W}|_b \right)$ is constant with respect to $b \in B$. (In particular, it is invariant even for not necessarily flat families of Lagrangian Quot schemes.)
\en
\end{lemma}

\begin{proof} (1) Choose two different general Hecke transforms $\widetilde{W_1}$ and $\widetilde{W_2}$ of $W$. We may assume that the Hecke transforms are obtained at distinct points $p_1, \dots, p_{t_1}$ and $q_1, \dots, q_{t_2}$, respectively. We can take a Hecke transform of $\widetilde{W_1}$ at appropriate Lagrangian subspaces of $(\widetilde{W_1})_{q_i} = W_{q_i}$ for $1 \le i \le t_2$, and also a Hecke transform of $\widetilde{W_2}$ at suitable Lagrangian subspaces of $(\tW_2)_{p_j} = W_{p_j}$ for $1 \le j \le t_1$ to obtain a symplectic bundle $\widetilde{W_3}$ which is a common Hecke transform of $\tW_1$ and $\tW_2$. By generality of the choices made, we may assume for $i \in \{ 1 , 2 \}$ that the Lagrangian Quot schemes of all the intermediate Hecke transforms between $\tW_i$ and $\tW_3$ also have property $\cP$. The desired equality
$$
N_{C,e}^{w + t_1 n}( P(\alpha) \Qt_{\rho_n}^{t_1}(\alpha) ; \tW_1) \: = \: N_{C,e}^{w + t_2 n}( P(\alpha) \Qt_{\rho_n}^{t_2}(\alpha);\tW_2)
$$
is obtained by applying Lemma \ref{Prop:Hecke-GW} successively, starting from the common Hecke transform $\tW_3$.

(2) For a given $b_0 \in B$, by Lemma \ref{lemma-heckeP1}, there exists $t \gg 0$ such that if $\widetilde{\scWbo}$ is a general Hecke transform along $t$ Lagrangian subspaces of $\scWbo$, then $LQ_e \left( \widetilde{\scWbo} \right)$ has property $\cP$. By openness of property $\cP$, there exists an open subset $U$ of the component of $B$ containing $b_0$, such that for each $b \in U$ and for a general symplectic Hecke transformation of $\scWb$ of degree $w + tn$, the scheme $LQ_e \left( \widetilde{\scWb} \right)$ has property $\cP$.

Thus, shrinking $U$ if necessary, we may choose a family of degree $w + tn$ Hecke transforms $\widetilde{\scW|_{U}} \to \cC|_{U} \to U$, all having property $\cP$. By Proposition \ref{prop: not depend on b}, we see that
\be \label{Ntilde constant} N^{w+tn}_{\cC_b , e} \left( P(\alpha) \Qt_{\rho_n}^t (\alpha) ; \widetilde{\scWb} \right) \hbox{ is constant with respect to } b \in U . \ee

Now let $b^\prime$ be any other point of $B$. As each component of $B$ is a quasi-projective curve, we can find a finite connected chain of open subsets $U = U_0 , U_1 , \ldots , U_\nu$ of components of $B$ with $b_0 \in U_0$ and $b^\prime \in U_\nu$, equipped with families of Hecke transforms
\[ \widetilde{\scW_j} \ \to \ \cC|_{U_j} \ \to \ U_j \]
of $\scW|_{U_j}$ of degree $w + t_j n$ as above such that $LQ_e \left( \widetilde{\scW_j}|_b \right)$ has property $\cP$ for each $b \in U_j$. Now the the numbers $t_j$ may be different, but for $b \in U_j \cap U_k$, by part (1) we have equality
\[ N^{w + t_j n}_{\cC_b , e} \left( P(\alpha) \cdot \Qt_{\rho_n}^{t_j}(\alpha) ; \widetilde{{\mathscr W}_j}|_b \right) \ = \ N^{w + t_k n}_{\cC_b , e} \left( P(\alpha) \cdot \Qt_{\rho_n}^{t_k}(\alpha) ; \widetilde{{\mathscr W}_k}|_b \right) . \]
	
By definition of $\tN^w_{C, e} ( P; W )$ and by (\ref{Ntilde constant}) it follows that $\tN^w_{\cC_b, e} ( P(\alpha) ; \scWb )$ is constant with respect to $b \in B$. \end{proof}

If $\LQ$ has property $\cP$, then in computing $\widetilde{N}_{C,e}^w ( P (\alpha) ; W )$ we can take $\tW = W$. Thus we obtain:

\begin{prop} \label{definitions coincide} Let $W$ be any symplectic bundle of degree $w$ such that $\LQ$ has property $\cP.$ Then we have
$$ \widetilde{N}_{C,e}^w ( P (\alpha) ; W ) \ = \ N_{C,e}^w ( P (\alpha) ; W ) . $$
In particular, the two definitions of intersection number coincide. \end{prop}


\noindent We shall shortly see that if $\LQ$ has property $\cP$, then $N^w_{C, e} ( P;W )$ enumerates Lagrangian subbundles of $W$ satisfying a certain condition.

\subsection{Relations between intersection numbers}

Here we study a behavior of the numbers $\widetilde{N}^w_{C, e} ( P ; W )$ under various transformations. Let $W$ be an $L$-valued symplectic bundle of degree $w$ over $C$. Let $\widehat{L}$ be a line bundle of degree $\hat{\ell}$ over $C$. Then $W \otimes \widehat{L}$ is an $L \otimes \widehat{L}^2$-valued symplectic bundle of degree $w + 2n \hat{\ell}$.

\begin{prop} \label{Prop:tranform-GW} Let $W$ and $\widehat{L}$ be as above. Then
$$ \widetilde{N}_{C,e}^{w} ( P (\alpha) ; W ) \ = \ \widetilde{N}_{C,e + n \hat{\ell}}^{w + 2n\hat{\ell}}( P(\alpha) ; W \otimes \widehat{L}) . $$
\end{prop}

\begin{proof} The proposition is immediate from the fact, already used in Lemma \ref{lemma-heckeP1}, that the association
\[ [ E \to W] \ \mapsto \ \left[ ( E \otimes \widehat{L} ) \to ( W \otimes \widehat{L} ) \right] \]
defines an isomorphism $LQ_e(W) \isom LQ_{e + n\hat{\ell}} ( W \otimes \widehat{L} )$. \end{proof}

\begin{prop} \label{prop:recursive} Let $W$ be an arbitrary symplectic bundle of degree $w$, and assume $\LQ$ is nonempty. Then for any integer $k \geq 0$, we have
\be \label{eq0} \widetilde{N}_{C,e}^w ( P (\alpha) ; W ) \ = \ \widetilde{N}_{C,e-nk}^w ( P (\alpha) \cdot \Qt_{\rho_n}^{2k}(\alpha) ; W ) . \ee
\end{prop}

\begin{proof} Firstly, by the definition of $\widetilde{N}_{C, e}^w ( P (\alpha) ; W)$, for large enough $m \gg 2k$ the left hand side of (\ref{eq0}) can be written as
\be \label{eq1} \widetilde{N}_{C,e}^w ( P (\alpha) ; W ) \ = \ N_{C,e}^{w+mn} ( P (\alpha) \cdot \Qt_{\rho_n}^m(\alpha) ; \tW ) \ee
for a general Hecke transform $W \subset \tW$ with $\deg(\tW) = w + mn$.

Now set $h := m - 2k$. Since $h$ is sufficiently large, the right hand side of (\ref{eq0}) can be written as
\be\label{eq2} \widetilde{N}_{C, e-nk}^w ( P (\alpha) \cdot \Qt_{\rho_n}^{2k}(\alpha) ; W ) \ = \ N_{C, e-nk}^{w+hn} ( P (\alpha) \cdot \Qt_{\rho_n}^{2k+h} (\alpha) ; \tW_1) \ee
for a general Hecke transform $W \subset \tW_1$ with $\deg(\tW_1) = w+nh$. Let $\widehat{L}$ be a line bundle of degree $k$. 
 Then by Proposition \ref{Prop:tranform-GW}, the right hand side of (\ref{eq2}) can in turn be written as
\begin{multline} \label{eq3}
N_{C, e - nk}^{w + hn} ( P (\alpha) \cdot \Qt_{\rho_n}^{2k + h} (\alpha) ; \tW_1) \ = \ N_{C,e}^{w + hn + 2nk}( P (\alpha) \cdot \Qt_{\rho_n}^{2k+h} (\alpha) ; \tW_1 \otimes \widehat{L} ) \\
 = \ N_{C,e}^{w + mn}( P (\alpha) \cdot \Qt_{\rho_n}^m (\alpha) ; \tW_1 \otimes \widehat{L} ) \end{multline}
since $m = h + 2k$. As both $LQ_e ( \tW )$ and $LQ_e ( \tW_1 \otimes \widehat{L} )$ have property $\cP$, by Lemma \ref{lemma: Ntilde} (2) the right hand sides of (\ref{eq1}) and (\ref{eq3}) coincide.
\end{proof}

\subsection{Counting Lagrangian subbundles}
Informally speaking, a Gromov--Witten invariant of a manifold $X$ gives a virtual count of certain curves inside $X$ with prescribed intersection properties. In this section, we shall prove that in fact the Gromov--Witten invariants $N_{C, e}^w (P ; W)$ really enumerate Lagrangian subbundles; the virtual count corresponds to an actual number.

\begin{prop} \label{Prop-Open part} Let $W$ be any symplectic bundle of degree $w = n \ell$ and $e$ an integer such that $LQ_e(W)$ has property $\cP$. Let $P(\alpha)$ be a weighted homogeneous polynomial of degree $D(n, e, \ell)$ which is of the form
$$ P(\alpha) \ = \ \left( \prod_{i=1}^s \alpha_{k_i} \right) \cdot \Qt_{\rho_n}^t (\alpha) . $$
Let $p_1, \ldots , p_s, q_1, \ldots , q_t \in C$ be distinct points. For each $p_i$, let $H_i \subset W_{p_i}$ be an isotropic subspace of dimension $n + 1 - k_i$, where $1 \le k_i \le n$. For each $q_j$, let $\Lambda_j \subset W_{q_j}$ be a Lagrangian subspace. Then the Gromov--Witten number $\widetilde{N}_{C,e}^w ( P ; W )$ enumerates the points, counted with multiplicities, of the intersection
\be \label{MovingLemmaIntersection} \left( \bigcap_{i=1}^s \bbX_{k_i}^\circ( \gamma_i \cdot H_i; p_i ) \right) \ \cap \ \left( \bigcap_{j=1}^t \bbX_{\rho_n}^\circ (\eta_j \cdot \Lambda_j ; q_j) \right) \ee
for a general choice of $\gamma_i \in \Sp \left( W_{p_i} \right)$ and $\eta_j \in \Sp \left( W_{q_j} \right)$.
\end{prop}

\begin{proof} By Corollary \ref{cor-class} and Remark \ref{remark:intersection}, the invariant $N_{C,e}^w ( P ; W)$ enumerates the points in the intersection
$$ \left( \bigcap_{i=1}^s \bbX_{k_i}( \gamma_i \cdot H_i; p_i ) \right) \ \cap \ \left( \bigcap_{j=1}^t \bbX_{\rho_n} (\eta_j \cdot \Lambda_j ; q_j) \right) $$
for a general choice of $\gamma_i \in \Sp (W_{p_i})$ and $\eta_j \in \Sp(W_{q_j})$. To see that this coincides with (\ref{MovingLemmaIntersection}), we show that all the intersection points lie inside $\LQo$.

Suppose $[ \psi \colon E \to W ]$ is a point of the intersection (\ref{MovingLemmaIntersection}) such that $E$ is nonsaturated, so that $\bE / E$ is a torsion sheaf of degree $r \ge 1$. For $1 \le i \le s$, we have maps $\bE_{p_i} \to \frac{W_{p_i}}{(\gamma_i \cdot H_i)^\perp}$. For those $i$ such that this is not surjective, $[E \to W]$ satisfies
\begin{equation} E \in f_r^{-1} \left( \bbY_{k_i} ( \gamma_i \cdot H_i ; p_i ) \right) , \label{conditionOneA} \end{equation}
where $f_r$ is as defined in {\S} \ref{NonSat} and $\bbY_{k_i} ( \gamma_i \cdot H_i ; p_i)$ as in Proposition \ref{prop:Chern}. By Kleiman's theorem, for general $\gamma_i$ this is a condition of codimension $k_i$ on each component of $LQ_{e+r}^\circ (W)$ (note that this may not be equidimensional).

On the other hand, for those $i$ such that $\bE_{p_i} \to \frac{W_{p_i}}{(\gamma_i \cdot H_i)^\perp}$ is surjective, we must have
\begin{equation} \dim \left( \psi ( E_{p_i} ) \cap (\gamma_i \cdot H_i)^\perp \cap \bE_p \right) \ \ge \ l_i - n + k_i \label{conditionOneB} \end{equation}
where $l_i$ is the rank of the linear map $E_{p_i} \to \bE_{p_i}$. By the proof of Proposition \ref{prop:Chern}, the condition (\ref{conditionOneB}) is of codimension at least $k_i$ on each fiber of $f_r \colon B_r \to LQ_{e+r}^\circ ( W )$.

Next, for $1 \le j \le t$ we consider the compositions $E_{q_j} \to \bE_{q_j} \to \frac{W_{q_j}}{\eta_j \cdot \Lambda_j}$. Write $m_j$ for the rank of $E_{q_j} \to \bE_{q_j}$ and $b_j := \dim ( \bE_{q_j} \cap \eta_j \cdot \Lambda_j )$. 
Now we claim that the Schubert cycle
\begin{equation} \{ \Pi \in \LG \left( W_{q_j} \right) \: | \: \dim ( \Pi \cap \Lambda_j ) \ = \ b_j \} \label{PiCapLambda} \end{equation}
is of dimension $b_j ( n - b_j )$. For; as $\Lambda_j$ is isotropic and of dimension $n$, the image of $\Pi \to W_{q_j} / \Lambda_j = \Lambda_j^\vee$ is exactly $\left( \Pi \cap \Lambda_j \right)^\perp$. Hence $\Pi \in \LG ( W_{q_j} )$ is determined by the choice of $\Pi \cap \Lambda_j \in \Gr ( b_j, \Lambda_j )$, which is a variety of dimension $b_j ( n - b_j )$. Therefore, (\ref{PiCapLambda}) has codimension $\frac{1}{2}n(n+1) - b_j ( n - b_j )$ in $\LG \left( W_{q_j} \right)$. Hence by Kleiman's theorem, for general $\eta_j \in \Sp \left( W_{q_j} \right)$ the locus of $F \in LQ_{e+r}^\circ (W)$ satisfying
\begin{equation} \dim ( F_{q_j} \cap \eta_j \cdot \Lambda_j ) \ = \ b_j \label{conditionTwo} \end{equation}
is either empty or of codimension $\frac{1}{2}n(n+1) - b_j ( n - b_j )$ in $LQ_{e+r}^\circ (W)$.

Furthermore, as in the proof of Proposition \ref{prop:Chern}, for fixed $\eta_j$ the condition that 
\begin{multline} \hbox{the image of } ( E_{q_j} \to \bE_{q_j} ) \hbox{ has dimension $m_j$ and} \\
\hbox{is contained in } \bE_{q_j} \cap \eta_j \cdot \Lambda_j \label{conditionThree} \end{multline}
defines a locus of dimension
\begin{multline*} \dim \Gr ( m_j , b_j ) + \dim \Quot^{0, r - (n - m_j)} ( \bE ) \ = \ m_j ( b_j - m_j ) + n ( r - n + m_j ) \\
 = \ \dim \Quot^{0, r} ( \bE ) - ( n^2 + m_j^2 - n m_j - m_j b_j ) \end{multline*}
in $\Quot^{0, r} ( \bE )$.

Now (\ref{conditionOneA}) and (\ref{conditionTwo}) are conditions purely on the base of $f_r \colon B_r \to LQ_{e+r}^\circ ( W )$. They are defined by pulling back Schubert varieties via the evaluation maps $\ev_x \colon LQ_{e+r}^\circ (W) \to \LG( W_x )$ (which are morphisms) for distinct points $x \in \{ p_1 , \ldots , p_s , q_1 , \ldots , q_t \}$. Thus an induction argument using Kleiman's theorem shows that the intersection of the loci defined on $LQ_{e+r}^\circ$ by (\ref{conditionOneA}) and (\ref{conditionTwo}) is either empty or of the expected codimension on each component of $LQ_{e+r}^\circ (W)$.

Next, (\ref{conditionOneB}) and (\ref{conditionThree}) are conditions purely on the fibers of $f_r \colon B_r \to LQ_{e+r}^\circ ( W )$. As the points $p_i$ and $q_j$ are all distinct, for general $\eta_j$ the loci defined by these conditions intersect properly in each fiber of $f_r$.

Putting all these together, to compute the codimension of (\ref{MovingLemmaIntersection}) for general $\gamma_1 , \ldots , \gamma_s$ and $\eta_1 , \ldots , \eta_t$, we can add the codimensions defined by the conditions (\ref{conditionOneA}), (\ref{conditionOneB}), (\ref{conditionTwo}) and (\ref{conditionThree}). We obtain a locus in $B_r$ which is empty or of codimension at least
\begin{multline*} \sum_{i=1}^s k_i + t \cdot \frac{1}{2}n(n+1) + \sum_{j = 1}^t \left( - b_j ( n - b_j ) + ( n^2 + m_j^2 - n m_j - m_j b_j ) \right) \\
 = \ \sum_{i=1}^s k_i + t \cdot \frac{1}{2}n(n+1) + \frac{1}{2} \cdot \sum_{j=1}^t \left( ( n - b_j)^2 + (n - m_j)^2 + (b_j - m_j)^2 \right) \\
\ge \ \sum_{i=1}^s k_i + t \cdot \frac{1}{2}n(n+1) \ = \ D (n, e, \ell) . \end{multline*}
But since $LQ_e (W)$ has property $\cP$, no $B_r$ is dense. Thus the intersection of (\ref{MovingLemmaIntersection}) with the nonsaturated locus is empty for general $\gamma_i$ and $\eta_j$, as desired. \end{proof}

\begin{cor} \label{Prop:trivial} Let $C$ be a smooth projective curve of genus $g$. Suppose $\lambda^{1} , \ldots , \lambda^{m} \in \Dn$ are strict partitions such that $\sum_{i=1}^m|\lambda^i| = D(n, e, 0)$. Set $P(\alpha) = \prod_{i=1}^m \Qt_{\lambda^i}(\alpha)$. Then we have the equality \begin{equation} \label{multiplication}
\widetilde{N}^{0}_{C,e}(P(\alpha); \Oc^{\oplus 2n}) \ = \ \left\langle \sigma_{\lambda^1} , \ldots\hspace{0.02in} , \sigma_{\lambda^m}\right\rangle_{g,|e|} .
\end{equation}
\end{cor}

\begin{proof}
If $P(\alpha) = \prod_{j=1}^s \alpha_{r_j}$ is a monomial in $\alpha_1 , \ldots , \alpha_n$ of weighted degree $D(n, e, 0)$, then by Proposition \ref{Prop-Open part} and the definition of the Gromov--Witten invariant in \S\:2, we obtain
\be
\label{eq5} \widetilde{N}_{C,e}^0 ( P(\alpha) ; \mathcal{O}_{C}^{\oplus 2n}  ) \ = \ \left\langle \sigma_{k_1}, \ldots ,\sigma_{k_s} \right\rangle_{g,|e|} .
\ee
On the other hand, the Vafa--Intriligator-type formula shows that the Gromov--Witten invariant $\left\langle \sigma_{\lambda^1}, \ldots ,\sigma_{\lambda^m}\right\rangle_{g,d}$ only depends on the product $\prod_{i=1}^m \sigma_{\lambda^i}$ of the arguments. Since $\sigma_1 , \ldots , \sigma_n$ generate $\Chow^*(\lg)$, the class $P(\sigma) := \prod_{i=1}^m \sigma_{\lambda^i}$ can be written as a sum of monomials in $\sigma_1 , \ldots , \sigma_n$. But since each $\sigma_i$ corresponds to $\alpha_i$ and hence $P(\sigma)$ to $P(\alpha)$, the desired equality follows from the linearity of both sides of the equality (\ref{eq5}).
\end{proof}

Corollary \ref{Prop:trivial} together with Proposition \ref{prop:recursive} yields the following recursive relation among Gromov--Witten invariants of $\lg$.

\begin{cor} Let $n > 0$ and $g, d \geq 0$ be given. Suppose $\sum_{i=1}^m |\lambda^i |= (n+1)d - \frac{n(n+1)}{2}(g-1)$. Then for any $k \geq 0$, we have
$$
\left\langle \sigma_{\lambda^1} , \ldots \hspace{0.02in} , \sigma_{\lambda^m} \right\rangle_{g,d} \ = \ \left\langle \sigma_{\rho_n}^{2k} , \sigma_{\lambda^1} , \ldots \hspace{0.02in} , \sigma_{\lambda^m} \right\rangle_{g, d + kn} .
$$
\end{cor}

\section{Main results}

From the discussion in the previous sections, we conclude:

\begin{theorem} \label{Theorem: intersection} Let $C$ be a smooth projective curve of genus $g$ and $W$ a symplectic bundle over $C$ of degree $w = n \ell$. Then for a polynomial $P(\alpha)$ of degree $D(n, e, \ell),$ the number $\widetilde{N}_{g,e}^w ( P(\alpha) ; W )$ is computed by
\begin{displaymath}\widetilde{N}_{C,e}^w ( P (\alpha) ; W ) \ =
\begin{cases}
 \scriptstyle{\scriptstyle{A \sum_{J\in \mathcal{I}_{n+1}^e}\big
\{S_{\rho_{n}}(\zeta^J)\big\}^{g-1} P(\zeta^J)}} & \text{if} \  \:\ell=2m,\\
\scriptstyle {\scriptstyle{A \sum_{J\in \mathcal{I}_{n+1}^e}\big
\{S_{\rho_{n}}(\zeta^J)\big\}^{g-1} {\Qt}_{\rho_n}(\zeta^J) P(\zeta^J)}}
& \text{if} \ \:  \ell=2m-1,
\end{cases}
\end{displaymath}
where $A := 2^{n(g-1)+e-mn}$ and $P(\zeta^J) := P \left( E_1(\zeta^J), \dots, E_n(\zeta^J) \right)$.
\end{theorem}

\begin{proof}
For the case $w = 2mn$, we take a line bundle $\Xi$ on $C$ of degree $-m$, so that $\tW := W \otimes \Xi$ is a symplectic bundle of degree $0$ over $C$. Then by Proposition \ref{Prop:tranform-GW}, we have
\[ \widetilde{N}_{C,e}^w ( P(\alpha) ; W ) \ = \ \widetilde{N}_{C,e-mn}^0 ( P(\alpha) ; \tW ) . \]
Thus the result follows from Propositions \ref{vafa-Int-g} and  \ref{Prop:trivial}.

If $w = (2m-1)n$, by Lemma \ref{Prop:Hecke-GW} we have
\[ \widetilde{N}_{C,e}^w ( P(\alpha) ; W ) \ = \ \widetilde{N}_{C,e}^{w+n} ( \Qt_{\rho_n}(\alpha) P(\alpha) ; W^H ) \]
for some Hecke transform $W^H$ of degree $w+n$. Since $w+n = 2mn$, we are reduced to the previous case. \end{proof}

Now let $P$ be the constant polynomial $1$, and assume $W$ is general (for example, very stable). Here the invariant $\widetilde{N}_{C,e}^w (1 ; W)$ is precisely the number of maximal Lagrangian subbundles of $W$. Recall from Lemma  \ref{lemma: Ntilde} that in this case $\widetilde{N}_{C,e}^w (1 ; W)$ depends only on the genus of $C$, so we denote it by $N (g, n, \ell, e)$. The following is immediate from Theorem \ref{Theorem: intersection}.

\begin{cor} \label{counting_formula} Let $W$ be a general stable symplectic  bundle over $C$ of rank $2n$ and degree $w = n \ell$, where $n(\ell - g + 1)$ is even. Let $e = \frac{1}{2}n(\ell - g + 1)$. Then the number $N (g, n, \ell, e)$ of maximal Lagrangian subbundles is given by
\begin{displaymath} N (g, n, \ell, e) \ = \
\begin{cases}
 \scriptstyle{{B_1} \sum_{J \in \mathcal{I}_{n+1}^e}\big
\{S_{\rho_{n}}(\zeta^J)\big\}^{g-1}} & \text{if} \ \: \ell = 2m , \\
\scriptstyle {\scriptstyle{{B_2} \sum_{J \in \mathcal{I}_{n+1}^e}\big
\{S_{\rho_{n}}(\zeta^J)\big\}^{g-1} {\Qt}_{\rho_n}(\zeta^J)}}
& \text{if} \ \: \ell = 2m-1 ,
\end{cases}
\end{displaymath}
where $B_1 = \mathlarger{\sqrt{2}}^{n(g-1)}$ and $B_2 = \mathlarger{\sqrt{2}}^{n(g-2)}$.
\end{cor}
Using this formula, we compute by hand the number of maximal Lagrangian subbundles of a general $W$ of rank $2n \le 4$.

\begin{cor} \label{examples} For  $g \ge 2$ and $e = \frac{1}{2}n(\ell - g + 1)$, we have the following.
\begin{enumerate}
\item $n = 1, \: \ell \not\equiv g \mod 2$: \ $N (g, 1, \ell, e) = 2^g$.
\item $n = 2, \: g $ even, $\ell$ odd: \ $N(g,2,-1,-g) = 2^{g-1} (3^g + 1)$.
\item $n = 2, \: g $ even, $\ell$ even: \ $N(g,2,0,-g+1) = 2^{g-1} (3^g - 1)$.
\item $n = 2, \: g $ odd, $\ell$ odd: \ $N(g,2,-1,-g) = 2^{g-1} (3^g - 1)$.
\item $n = 2, \: g $ odd, $\ell$ even: \ $N(g,2,0,-g+1) = 2^{g-1} (3^g + 1)$.
\end{enumerate}
\end{cor}

\begin{remark} In (1), the number $2^g$ coincides with the number of maximal line subbundles of a general rank 2 vector bundle obtained in \cite{Se} and \cite{Na}. This can be explained by the fact that any rank 2 vector bundle $V$ has a symplectic structure given by $V \cong V^\vee \otimes \det(V)$, and any line subbundle is Lagrangian. \end{remark}

\begin{remark} By Holla \cite[Theorem 4.2]{Ho}, if $g = 2$, the number of maximal rank 2 subbundles of a general rank 4 vector bundle $V$ is 24 (resp., 40), if $\deg (V) \equiv 2 \mod 4$ (resp., $\deg (V) \equiv 0 \mod 4)$. These can be compared with the numbers 20 and 16 given by (2) and (3) respectively.

It should be noted that \cite[Theorem 2]{Hit2}, in our language, states incorrectly that $N(2,2,0,-1) = 24$. This is due to a mistake in the geometric argument on \cite[p.\ 270]{Hit2}. The correct statement of \cite[Theorem 2]{Hit2} is that the moduli map $\Phi$ is surjective and generically finite of degree $20$. \end{remark}


\begin{thebibliography}{99}

\bibitem{Ab} L.\ Abrams: \textsl{The quantum Euler class and the quantum cohomology of the Grassmannians}, Israel J.\ Math.\ \textbf{117} (2000), 335--352.

\bibitem{Be} A.\ Bertram: \textsl{Towards a Schubert calculus for maps from a Riemmann surface to a Grassmannian}, Internat.\ J.\ Math.\ \textbf{5} (1994), no.\ 6, 811--825.

\bibitem{BG} I.\ Biswas and T.\ G\'{o}mez: \textsl{Hecke correspondence for symplectic bundles with application to the Picard bundles}, Internat.\ J.\ Math.\ \textbf{17} (2006), no.\ 1, 45--63.

\bibitem{BR} I.\ Biswas and S.\ Ramanan: \textsl{An infinitesimal study of the moduli of Hitchin pairs}, J.\ London Math.\ Soc.\ \textbf{49} (1994), no.\ 2, 219--231.

\bibitem{CMP} P.\ E.\ Chaput, L.\ Manivel and N.\ Perrin: \textsl{Quantum cohomology of minuscule homogeneous spaces III, Semi-simplicity and consequences}, Canad.\ J.\ Math.\ \textbf{62} (2010), no.\ 6, 1246--1263.

\bibitem{Che1} D.\ Cheong: \textsl{Vafa--Intriligator type formulas and quantum Euler classes for Lagrangian and orthogonal Grassmannians}, Internat.\ J.\ Algebra Comput.\ \textbf{21} (2011), no.\ 4, 575--594.

\bibitem{Che2} D.\ Cheong: \textsl{Quantum multiplication operators for Lagrangian and orthogonal Grassmannians}, J.\ Algebr.\ Comb.\ \textbf{45} (2017), no.\ 4, 1153--1171.

\bibitem {CCH} D.\ Cheong, I.\ Choe and G.\ H.\ Hitching: \textsl{Irreducibility of Lagrangian Quot schemes over an algebraic curve}, arXiv:1804.00052.

\bibitem {CH2} I.\ Choe and G.\ H.\ Hitching: \textsl{Lagrangian subbundles of symplectic vector bundles over a curve}, Math.\ Proc.\ Camb.\ Phil.\ Soc.\ \textbf{153} (2012), 193--214.



\bibitem{Fu} W.\ Fulton: \textsl{Intersection Theory}. Springer-Verlag, Newyork, 1998.

\bibitem{FuPr} W.\ Fulton and P.\ Pragacz: \textsl{Schubert Varieties and Degeneracy Loci}. \textrm{Lecture Notes in Math.} 1689, Springer-Verlag, Berlin, 1998.



\bibitem {Hit2} G.\ H.\ Hitching: \textsl{Moduli of rank 4 symplectic vector bundles over a curve of genus 2}, J.\ London Math.\ Soc.\ (2) \textbf{75}  (2007), no.\ 1, 255--272.

\bibitem {Ho} Y.\ I.\ Holla: \textsl{Counting maximal subbundles via Gromov--Witten invariants}, Math.\ Ann.\ \textbf{328} (2004), no.\ 1--2, 121--133.

\bibitem {Kle} S.\ L.\ Kleiman: \textsl{The transversality of a general translate}, Compos.\ Math.\ \textbf{28} (1974), 287--297.

\bibitem {Ko} J.\ Koll\'{a}r: \textsl{Rational curves on algebraic varieties}, Springer-Verlag, 1996.

\bibitem{KT2002} A.\ Kresch and H.\ Tamvakis: \textsl{Double Schubert polynomials and degeneracy loci for the classical groups}, Ann.\ Inst.\ Fourier \textbf{52} (2002), no.\ 6, 1681--1727.

\bibitem{KT} A.\ Kresch and H.\ Tamvakis, \textsl{Quantum cohomology of the Lagrangian Grassmannian}, J.\ Algebraic Geometry \textbf{12} (2003),
777--810.

\bibitem {LN} H.\ Lange and P.\ E.\ Newstead: \textsl{Maximal subbundles and Gromov--Witten invariants}, A tribute to C.\ S.\ Seshadri (Chennai, 2002), 310--322, Trends Math.\, Birkh\"{a}user, Basel, 2003.


\bibitem{MO} A.\ Marian and D.\ Oprea: \textsl{Virtual intersections on the Quot scheme and Vafa--Intriligator formulas}, Duke Math.\ J.\ \textbf{136} (2007), no.\ 1, 81--113.

\bibitem {Na} M.\ Nagata: \textsl{On self-intersection number of a section on a ruled surface}, Nagoya Math.\ J.\ \textbf{37} (1970), 191--196.

\bibitem {Ox} W.\ M.\ Oxbury: \textsl{Varieties of maximal line subbundles}, Math.\ Proc.\ Camb. Philos.\ Soc.\ \textbf{129} (2000), no.\ 1, 9--18.

\bibitem {PoRo} M.\ Popa and M.\ Roth: \textsl{Stable maps and Quot schemes}, Invent.\ Math.\ \textbf{152} (2003), no.\ 3, 625--663.

\bibitem{PrRa} P.\ Pragacz and J.\ Ratajski: \textsl{Formulas for Lagrangian and orthogonal degeneracy loci; $\Qt$-polynomial approach}, Compos.\ Math.\ \textbf{107} (1997), 11--87.

\bibitem{RuTi} Y. Ruan; G.\ Tian: \textsl{A mathematical theory of quantum cohomology}, J.\  Differential Geom.\ \textbf{42} (1995), no.\ 2, 259--367.

\bibitem {Se} C.\ Segre: \textsl{Recherches g\'en\'erales sur les courbes et les surfaces alg\'ebriques}, Math.\ Ann.\ \textbf{34} (1889), 1--29.

\bibitem{Te} M.\ Teixidor i Bigas: \textsl{Subbundles of maximal degree}, Math.\ Proc.\ Camb.\  Phil.\ Soc.\ \textbf{136} (2003),  no.\ 3, 541--543.

\end{thebibliography}
\end{document}